\documentclass{article}

\usepackage{amsmath,amssymb,amsthm,mathrsfs}
\usepackage[overload]{empheq}
\usepackage{graphicx}
\usepackage[colorlinks=true]{hyperref}
\usepackage{pdfsync}
\usepackage{tikz}
\usetikzlibrary{decorations.pathreplacing}
\usetikzlibrary{patterns}

\newtheorem{theorem}{Theorem}
\newtheorem{proposition}{Proposition}
\newtheorem{lemma}{Lemma}

\theoremstyle{definition}\newtheorem{example}{Example}
\theoremstyle{definition}\newtheorem{definition}{Definition}
\theoremstyle{definition}\newtheorem{remark}{Remark}

\newcommand{\fonction}[5]{\begin{array}[t]{lrcl}#1 :&#2 &\longrightarrow &#3\\&#4& \longmapsto &#5 \end{array}}

\def\;{\,}

\def\di{\displaystyle}

\def\R{\mathbb{R}}
\def\N{\mathbb{N}}

\def\B{\mathrm{B}}
\def\BB{\overline{\B}}
\def\KK{\mathcal{K}_\Omega}
\def\C{\mathrm{C}}
\def\AC{\mathrm{AC}}
\def\L{\mathrm{L}}
\def\H{\mathrm{H}}
\def\E{\mathrm{E}}
\def\ID{\mathrm{Id}}

\def\t{\tau}
\def\a{\alpha}
\def\aa{\overline{\a}}
\def\aaa{\underline{\a}}
\def\c{\mathrm{c}}
\def\loc{\mathrm{loc}}

\def\I{\mathrm{I}}
\def\Im{\mathrm{I}^\a_{a+}}
\def\Ip{\mathrm{I}^\a_{b-}}

\def\II{\mathbb{I}}
\def\IIm{\II_{a+}}
\def\IImf{\II_{a+}^f}
\def\IIp{\II_{b-}}

\def\CD{{}_\c \mathrm{D}}
\def\CDm{{}_\c \mathrm{D}^\alpha_{a+}}
\def\CDp{{}_\c \mathrm{D}^\alpha_{b-}}
\def\D{\mathrm{D}}
\def\Dm{\mathrm{D}^\alpha_{a+}}
\def\Dp{\mathrm{D}^\alpha_{b-}}

\title{Cauchy-Lipschitz theory for fractional multi-order dynamics -- State-transition matrices, Duhamel formulas and duality theorems}

\author{Lo\"ic Bourdin\footnote{Universit\'e de Limoges, Institut de recherche XLIM, D\'epartement de Math\'ematiques et d'Informatique. UMR CNRS 7252. Limoges, France (\texttt{loic.bourdin@unilim.fr}).}
}

\date{}

\begin{document}

\maketitle

\begin{abstract}
The aim of the present paper is to contribute to the development of the study of Cauchy problems involving Riemann-Liouville and Caputo fractional derivatives. Firstly existence-uniqueness results for solutions of non-linear Cauchy problems with vector fractional multi-order are addressed. A qualitative result about the behavior of local but non-global solutions is also provided. Finally the major aim of this paper is to introduce notions of \textit{fractional state-transition matrices} and to derive fractional versions of the classical Duhamel formula. We also prove duality theorems relying left state-transition matrices with right state-transition matrices.
\end{abstract}

\bigskip


\bigskip


\tableofcontents

\section{Introduction}

The \textit{fractional calculus} is the mathematical field that deals with the generalization of the classical notions of integral and derivative to any real order. The fractional calculus seems to be originally introduced in 1695 in a letter written by Leibniz to L'Hospital where he suggested to generalize his celebrated formula of the $k^\textrm{th}$-derivative of a product (where $k \in \N^*$ is a positive integer) to any positive real $k > 0$. In another letter to Bernoulli, Leibniz mentioned derivatives of \textit{general order}. Since then, numerous renowned mathematicians introduced several notions of fractional operators. We can cite the works of Euler (1730's), Fourier (1820's), Liouville (1830's), Riemann (1840's), Sonin (1860's), Gr\"unwald (1860's), Letnikov (1860's), Caputo (1960's), etc. All these notions are not disconnected. In most cases it can be proved that two different notions actually coincide or are correlated by an explicit formula. 

\smallskip

For a long time, the fractional calculus was only considered as a pure mathematical branch. In 1974, a first conference dedicated to this topic was organized by Ross at the University of New Haven (Connecticut, USA). Since then, the fractional calculus and its applications experience a boom in several scientific fields. The uses are so varied that it seems difficult to give a complete overview of the current researches involving fractional operators. We can at least mention that the fractional calculus is widely applied in the physical context of anomalous diffusion, see \textit{e.g.} \cite{hilf1,metz,neel,neel2,oldh2,zasl,zasl2}. Due to the non-locality of the fractional operators, they are also used in order to take into account of memory effects, see \textit{e.g.} \cite{bagl2,bagl3,dubo,pfit} where {viscoelasticity} is modelled by a fractional differential equation. We also refer to studies in wave mechanic \cite{alme}, economy \cite{comt}, biology \cite{gloc,magi}, acoustic \cite{mati2}, thermodynamic \cite{hilf2}, probability \cite{levy}, etc. In a more general point of view, fractional differential equations are even considered as an alternative model to non-linear differential equations, see \cite{boni}. We refer to \cite{hilf3,saba} for a large panorama of applications of fractional calculus.

\smallskip

The first reference book \cite{oldh} on fractional calculus, developing some mathematical aspects and applications, was written by Oldham and Spanier in 1974. In 1993, Miller and Ross~\cite{mill} have studied fractional differential equations. The monographs \cite{samk} of Kilbas, Marichev and Samko in 1987 and \cite{kilb} of Kilbas, Srivastava and Trujillo in 2006 are essential books on fractional calculus, dealing with mathematical aspects with rigorous proofs, in particular concerning regularity issues, with fractional differential equations and containing some applications. We also refer to \cite{dubo,iniz,podl} and some chapters of \cite{gore,hilf,mati} for handy introductions to fractional calculus. Finally, we also mention \cite{mach} for the recent history of the fractional calculus. 

\bigskip

The aim of the present paper is to contribute to the development of the study of Cauchy problems involving Riemann-Liouville and Caputo fractional derivatives, providing some new results of different types. Section~\ref{seccauchyproblems} is devoted to existence-uniqueness results for solutions of fractional Cauchy problems. We also prove a qualitative result concerning the behavior of local but non-global solutions. Section~\ref{sectransition} is devoted to the introduction of \textit{fractional state-transition matrices} and to fractional versions of the classical Duhamel formula. We also prove duality theorems relying left state-transition matrices with right state-transition matrices. But, before detailing the contributions of these two sections, we feel that it is of interest to give first a brief overview of the existing results in the literature. 

\paragraph*{Brief overview on the existing fractional Cauchy-Lipchitz theory.}
The present paragraph is widely inspired by the survey \cite{kilb2} and by \cite[Chapter~3]{kilb}. Most of the investigations about fractional differential equations are concerned with the Riemann-Liouville fractional derivative $\Dm$. Precisely the usual non-linear Cauchy problem investigated has the form
\begin{equation}\label{appCeqRLdiffintro}
\left\lbrace \begin{array}{l}
\Dm [q](t) = f(q(t),t),  \\
\I^{1-\a}_{a+} [q](a) = q_a,
\end{array}
\right.
\end{equation}
considered on a compact interval $[a,b]$ with $a < b$, and with a fractional order $0 < \a < 1$. Essentially (and as in the classical theory), the investigations of the above fractional Cauchy problem are based on the integral formulation
\begin{equation}\label{appCeqRLintintro}
q(t) = \dfrac{1}{\Gamma(\a)}(t-a)^{\a-1} q_a + \Im [f(q,\cdot)](t).
\end{equation}
The first paper dealing with this topic is due to Pitcher and Sewell \cite{pitc} in 1938. They investigate the case where $q_a=0$ and $f$ is a continuous function satisfying a boundedness and a global Lipschitz continuity assumptions. Despite that Pitcher and Sewell present the original idea of reducing the differential problem into an integral one, their main result, providing the existence of a global continuous solution of the integral equation~\eqref{appCeqRLintintro}, is based on an erroneous proof. However, under the same kind of assumptions on $f$ (but without $q_a = 0$), Al-Bassam \cite{alba} uses the method of successive approximations in 1965 in order to well establish the existence of a global continuous solution of the integral equation \eqref{appCeqRLintintro}. Nevertheless, the hypotheses on $f$ (in particular the boundedness) are very strong and avoid to apply this result to the academic example $f(x,t) = x$. In 1996, Delbosco and Rodino \cite{delb} consider an initial condition of type $q(a) = q_a $ instead of $\I_{a+}^{1-\a} [q](a) =q_a$. Under some continuity assumption on $f$ and using a fixed point theorem, they prove that the fractional Cauchy problem admits at least a local continuous solution. This result corresponds to a fractional version of the classical Peano theorem. Under a global Lipschitz continuity assumption, they moreover prove that the solution is unique and global. Note that Hayek \textit{et al} \cite{haye} apply the same argument and obtain the same last result for the more usual initial condition $\I_{a+}^{1-\a} [q](a) =q_a$. Recall that Kilbas \textit{et al} establish existence-uniqueness results in spaces of integrable functions \cite{kilb3} and in weighted spaces of continuous functions \cite{kilb4}. Actually, the subject is widely treated in several directions. We can cite \cite{diet,idcz,mill} for other examples of studies.

\smallskip

As mentioned in \cite[Chapter~3]{kilb}, the differential equations involving Caputo fractional derivatives have not been studied extensively. In a first period, only particular cases have been investigated in the view of giving explicitly the exact solutions, see \textit{e.g.} the works of Gorenflo \textit{et al} in \cite{gore,gore2,gore3}. In 2002, Diethelm and Ford \cite{diet2} study the usual non-linear Cauchy problem involving the Caputo fractional derivative $\CDm$ given by
\begin{equation}\label{appCeqCdiffintro}
\left\lbrace \begin{array}{l}
\CDm [q](t) = f(q(t),t),  \\
q(a) = q_a, 
\end{array}
\right.
\end{equation}
considered on a compact interval $[a,b]$ with $a < b$, and with a fractional order $0 < \a < 1$. They prove the existence and uniqueness of a local continuous solution under the assumptions of continuity and local Lipschitz continuity of $f$. They also investigate the dependence of the solution with respect to the initial condition and to the function $f$. Kilbas and Marzan \cite{kilb5,kilb6} also study the above fractional Cauchy problem via its integral formulation
\begin{equation}\label{appCeqCintintro}
q(t) = q_a + \Im [f(q,\cdot)](t),
\end{equation}
and prove existence and uniqueness of a global continuous solution in the case of continuous and global Lipschitz continuous function $f$. In \cite{diet3}, the authors address the very interesting question concerning the possibility (or not) of two intersecting solutions to the equation $\CDm [q](t) = f(q(t),t)$. In the classical case $\a =1$ it is well-known that the answer is no, however the study seems to be much more complex in the fractional case $0 < \a < 1$. We also mention the work of Kilbas \textit{et al} \cite{kilb} investigating the issue of boundary condition at any $t \in [a,b]$ (\textit{i.e.} not necessarily at $t=a$).  

\smallskip

Numerous studies have also been devoted to existence-uniqueness results for differential equations involving other notions of fractional operators. For example, we can cite the study \cite{kilb7} with Hadamard fractional derivatives.

\paragraph*{Contributions of Section~\ref{seccauchyproblems}.}
The present paper is actually motivated by the needs of completing the existing fractional Cauchy-Lipschitz theory in order to investigate non-linear control systems involving Caputo fractional derivatives, and more precisely in order to derive a fractional version of the classical Pontryagin maximum principle in optimal control theory.\footnote{Work in progress. We mention here that first fractional versions of the Pontryagin maximum principle are addressed in~\cite{ali,kamo}.} 

\smallskip

Section~\ref{seccauchyproblems} is devoted to a general Cauchy-Lipschitz theory involving Riemann-Liouville and Caputo fractional derivatives that generalizes the basic notions and results of the classical theory surveyed \textit{e.g.} in \cite{codd,smal}. Namely, we will study the fractional Cauchy problems~\eqref{appCeqRLdiffintro} and \eqref{appCeqCdiffintro} in the following framework:
\begin{itemize}
\item The dynamic $f$ is a general Carath\'eodory function (not necessarily continuous in its second variable). Such a framework is crucial in order to deal with fractional control systems where controls can be discontinuous.
\item The fractional Cauchy problems are considered on a general interval with lower bound $a$ (\textit{i.e.} the interval is not necessarily compact). Such a framework is crucial in order to deal with free final time optimal control problems (\textit{e.g.} minimal time problems).
\item The trajectories $q$ are multidimensional, that is, with values in $\R^m$ (with $m \in \N^*$ a positive integer) and $\a$ is a vector fractional multi-order in the sense that $\a = (\a_i) \in (0,1]^m$. Such a framework is crucial in fractional optimal control theory in order to be able to rewrite a Bolza or a Lagrange cost functional into a Mayer cost functional. Indeed, this classical tricky transformation makes arise in the fractional framework a vector fractional multi-order $\a = (\a_i) \in (0,1]^m$.
\item The trajectories $q$ are with values in a nonempty open subset $\Omega \subset \R^m$. We prove in Theorem~\ref{thm2} that any local solution of~\eqref{appCeqCdiffintro} that is not global must go out of any compact subset of $\Omega$. This result is crucial in optimal control theory in order to prove stability results on the trajectories. In particular it allows to prove that the admissibility of a trajectory is stable under small $\L^1$-perturbations on the control.
\end{itemize}
With the above considerations, we give in Section~\ref{seccauchyproblems} integral representations for solutions of~\eqref{appCeqRLdiffintro} and \eqref{appCeqCdiffintro} (see Propositions~\ref{prop1RL} and \ref{prop1C}) and existence-uniqueness results (see Theorems~\ref{thm1RL}, \ref{thm1} and \ref{thm3}). As mentioned and referenced in the previous paragraph, similar results are already well-known in the literature. The originality here lies in the fact that we deal with a general interval (that is not necessarily compact) and with a vector fractional multi-order $\a = (\a_i) \in (0,1]^m$. The usual proofs have been extended to this case, and the details are provided in Appendices~\ref{appsectionCauchyRL} and \ref{appsectionCauchyC} for the reader's convenience. Nevertheless, we also prove in Section~\ref{seccauchyproblems} that any local solution of~\eqref{appCeqCdiffintro} that is not global must go out of any compact subset of $\Omega$ (see Theorem~\ref{thm2}). To the best of our knowledge, this result has not been addressed in the literature yet and, as above explained, should have many applications in stability theory of fractional control systems.

\paragraph*{Classical state-transition matrices and classical Duhamel formula.}
In this section we give basic recalls about state-transition matrices and the Duhamel formula in the classical case $\a = 1$. Let $a < b$, $m \in \N^*$ be a positive integer and $A : [a,b] \to \R^{m \times m}$ be a square matrix function. For any $s \in [a,b]$, we denote by $Z(\cdot,s) : [a,b] \to \R^{m \times m}$ the unique solution of the homogeneous linear matrix Cauchy problem given by
\begin{equation*}
\left\lbrace \begin{array}{l}
\dot{Z}(t) = A(t) \times Z(t),  \\
Z(s) = \ID_m. 
\end{array}
\right.
\end{equation*}
The function $Z(\cdot,\cdot)$ is the so-called \textit{state-transition matrix} associated to $A$. In the case where $A(\cdot) = A$ is constant, it is well-known that $Z(t,s)$ can be expressed as the exponential matrix $e^{A(t-s)}$. An explicit expression of the unique solution $q$ of the forward non-homogeneous linear vector Cauchy problem given by
\begin{equation*}
\left\lbrace \begin{array}{l}
\dot{q}(t) = A(t) \times q(t) + B(t),  \\
q(a) = q_a,
\end{array}
\right.
\end{equation*}
where $B : [a,b] \to \R^m$ is a vector function, can be derived and is well-known as the classical \textit{Duhamel formula} given by
$$ q(t) = Z(t,a) \times q_a + \int_a^t Z(t,s) \times B(s) \; ds. $$
Finally, it is also well-known that $Z(\cdot,\cdot)$ satisfies a \textit{duality property}. Precisely, for any $t \in [a,b]$, $Z(t,\cdot) : [a,b] \to \R^{m \times m}$ is the unique solution of the homogeneous linear matrix Cauchy problem given by
\begin{equation*}
\left\lbrace \begin{array}{l}
\dot{Z}(s) = -Z(s) \times A(s),  \\
Z(t) = \ID_m. 
\end{array}
\right.
\end{equation*}
As a consequence, an explicit expression of the unique solution $q$ of the following backward non-homogeneous linear vector Cauchy problem
\begin{equation*}
\left\lbrace \begin{array}{l}
\dot{q}(s) = - A(s)^\top \times q(s) - B(s),  \\
q(b) = q_b,
\end{array}
\right.
\end{equation*}
is given by
$$ q(s) = Z(b,s)^\top \times q_b + \int_s^b Z(t,s)^\top \times B(t) \; dt.  $$
The above duality property is crucial in optimal control theory in order to fully justify the definition of the backward adjoint vector with respect to the forward variation vectors.\footnote{Actually the generalization of the duality property to the fractional case, in order to fully justify the definition of the backward adjoint vector in fractional optimal control problems, is at the origin of the present paper.}

\paragraph*{Contributions of Section~\ref{sectransition}.}
Section~\ref{sectransition} is the major and the most original part of the present paper. To the best of our knowledge, all results presented in this section are not addressed in the literature yet. We introduce in Section~\ref{sectransition} the notions of \textit{Riemann-Liouville and Caputo fractional state-transition matrices} denoted respectively by $Z(\cdot,\cdot)$ and ${}_\c Z(\cdot,\cdot)$, see Definitions~\ref{defRLstate} and \ref{defCstate}. They are associated to a square matrix function $A(\cdot) \in \R^{m \times m}$ and to a matrix fractional multi-order $\a = (\a_{ij}) \in (0,1]^{m \times m}$. In the case where $\a$ is a row constant matrix, we prove fractional versions of the classical Duhamel formula (see Theorems~\ref{thmduhamelRL} and \ref{thmduhamelC}). 

\smallskip

We mention here that fractional Duhamel formulas are already obtained in \cite{das,idcz}, but only for constant square matrix functions $A(\cdot) = A$ and with a (uni-)order $\a \in (0,1]$. In this particular case, the authors of \cite{das,idcz} interestingly express the state-transition matrices as follows: 
$$ Z(t,s) = (t-s)^{\a-1}E_{\a,\a} (A(t-s)^\a) \quad \text{and} \quad {}_\c Z(\cdot,\cdot) = E_{\a,1}(A(t-s)^\a) , $$
where $E_{\a,\beta} $ denotes the classical Mittag-Leffler function. However, the square matrix functions involved in the definitions of variation vectors in fractional optimal control theory are not constant in general, and thus the generalization of the previous results to the non-constant case $A(\cdot) \in \R^{m \times m}$ reveals interests. 

\smallskip

Finally, we prove in Section~\ref{sectransition} duality theorems (see Theorems~\ref{thmduality} and \ref{thmduality2}) that generalize the duality property mentioned in the previous paragraph. These last results state that the left state-transition matrices associated to $A(\cdot) \in \R^{m \times m}$ and to a row constant matrix fractional multi-order $\a \in (0,1]^{m \times m}$ coincide with the right state-transition matrices associated to $A$ and to $\a^\top$, where $\a^\top$ denotes the column constant transpose of $\a$. %Note that this last result holds is proved in the Caputo case only if $A (\cdot) = A$ is constant.

\bigskip

Before detailing our results in Sections~\ref{seccauchyproblems} and \ref{sectransition}, we first give basic recalls on fractional calculus in Section~\ref{sectionbasics}. All proofs of Sections~\ref{seccauchyproblems} and \ref{sectransition} are detailed in Appendices~\ref{appsectionCauchyRL}, \ref{appsectionCauchyC} and \ref{appsectransition}.

\section{Basics on fractional calculus}\label{sectionbasics}
Throughout the paper, the notation $\N^*$ stands for the set of positive integers and the abbreviation R-L stands for Riemann-Liouville. This section is devoted to basic recalls about R-L and Caputo fractional operators. All definitions and results of Section~\ref{secbasics1} are very usual and are mostly extracted from the monographs~\cite{kilb,samk}. In Section~\ref{secbasics1} we focus on R-L and Caputo derivatives of onedimensional functions $q (\cdot ) \in \R$ with a fractional (uni-)order $\a \in [0,1]$. Sections~\ref{secbasics2} and \ref{secmultiorder} are devoted to the generalization of these notions to matrix functions $A (\cdot) \in \R^{m \times n}$ with matrix fractional multi-order $\a \in [0,1]^{m \times n}$, where $m$, $n \in \N^*$. A similar generalization was already considered in the literature, see \textit{e.g.} \cite{daft,lizh}.

\bigskip

We first introduce some notations available throughout the paper. Let $I \subset \R$ be an interval with a nonempty interior and let $m \in \N^*$ be a positive integer. We denote by:
\begin{itemize}
\item $\L^1(I,\R^m)$ the classical Lebesgue space of integrable functions defined on $I$ with values in $\R^m$, endowed with its usual norm $\Vert \cdot \Vert_1$;
\item $\L^\infty(I,\R^m)$ the classical Lebesgue space of essentially bounded functions defined on $I$ with values in $\R^m$;
\item $\C(I,\R^m)$ the classical space of continuous functions defined on $I$ with values in $\R^m$, endowed with the classical uniform norm $\Vert \cdot \Vert_\infty$;
\item $\AC(I,\R^m)$ the classical subspace of $\C(I,\R^m)$ of all absolutely continuous functions defined on $I$ with values in $\R^m$;
\item $\H^\lambda(I,\R^m)$ the classical subspace of $\C(I,\R^m)$ of all $\lambda$-Holderian continuous functions defined on $I$ with values in $\R^m$, where $\lambda \in (0,1]$.
\end{itemize}
Let us consider $\E(I,\R^m)$ one of the above space. We denote by $\E_\loc(I,\R^m)$ the set of all functions $q : I \to \R^m$ such that $q \in \E(J,\R^m)$ for every compact subinterval $J \subset I$. Let us consider $\E(I,\R^m)$ one of the three last above spaces and let $a \in I$. In that case, we denote by $\E_a (I,\R^m)$ the set of all functions $q \in \E(I,\R^m)$ such that $q(a) = 0$.

\subsection{Classical definitions and results in the scalar case}\label{secbasics1}
In this section we fix $a \in \R$ and $I \in \IIm$ where 
$$ \IIm := \{ I \subset \R \text{ interval such that } \{ a \} \varsubsetneq I \subset [a,+\infty) \}. $$
Note that $I$ is not necessarily compact. Precisely $I$ can be written either as $I = [a,+\infty)$, or as $I = [a,b)$ with $b >a$, or as $I= [a,b]$ with $b >a$. 

\begin{definition}
The left R-L fractional integral $\Im [q]$ of order $\a > 0$ of $q \in \L^1_\loc (I,\R)$ is defined on $I$ by
$$ \Im [q](t) := \int_a^t \dfrac{1}{\Gamma(\a)} (t-\t)^{\a-1} q(\t) \; d\t, $$
provided that the right-hand side term exists. For $\alpha = 0$ and $q \in \L^1_\loc (I,\R)$, we define $\I^0_{a+}[q] := q$.
\end{definition}

\noindent If $\a \geq 0$ and $q \in \L^1_\loc (I,\R)$, then $\Im [q](t)$ is defined for almost every $t \in I$.

\smallskip

\noindent For the next results, we refer to \cite[Lemma 2.1 p.72]{kilb} (Propositions~\ref{prop1} and \ref{prop2}), to \cite[Theorem 3.6 p.67]{samk} (Propositions~\ref{prop3} and \ref{prop4}) and to \cite[Lemma 2.3 p.73]{kilb} (Proposition~\ref{prop5}).

\begin{proposition}\label{prop1}
If $\a \geq 0$ and $q \in \L^1_\loc (I,\R)$, then $\Im [q] \in \L^1_\loc (I,\R)$. 
\end{proposition}

\begin{proposition}\label{prop2}
If $\a \geq 0$ and $q \in \L^\infty_\loc (I,\R)$, then $\Im [q] \in \L^\infty_\loc (I,\R)$. 
\end{proposition}

\noindent Let $\a \geq 0$ and $q \in \L^1_\loc (I,\R)$. Throughout the paper, if $\Im [q]$ is equal almost everywhere on $I$ to a continuous function on $I$, then $\Im [q]$ is automatically identified to its continuous representative. In that case, $\Im [q](t)$ is defined for every $t \in I$.

\begin{proposition}\label{prop3}
If $\a > 0$ and $q \in \L^\infty_\loc (I,\R)$, then $\I^{\alpha}_{a+} [ q ] \in \H^{\min(\a,1)}_{a,\loc}(I,\R) \subset \C_a(I,\R)$. 
\end{proposition}

\begin{proposition}\label{prop4}
If $0 < \a \leq 1$ and $q \in \L^\infty (I,\R)$, then $\I^{\alpha}_{a+} [ q ] \in \H^{\a}_{a}(I,\R) \subset \C_a(I,\R)$. 
\end{proposition}

\begin{proposition}\label{prop5}
If $\a_1 \geq 0$, $\a_2 \geq 0$ and $q \in \L^1_\loc (I,\R)$, then 
$$ \I^{\alpha_1}_{a+} \Big[ \I^{\alpha_2}_{a+} [ q ] \Big] = \I^{\alpha_1 +\alpha_2}_{a+} [ q ], $$
almost everywhere on $I$. If moreover $q \in \L^\infty_\loc (I,\R)$ and $\a_1 + \a_2 > 0$, the above equality is satisfied everywhere on $I$.
\end{proposition}

\begin{definition}
We say that $q \in \L^1_\loc (I,\R)$ possesses on $I$ a left R-L fractional derivative $\Dm [q]$ of order $0 \leq \a \leq 1$ if and only if $\I^{1-\a}_{a+} [q] \in \AC_\loc (I,\R)$. In that case $\Dm [q]$ is defined by
$$ \Dm [q](t) := \dfrac{d}{dt} \Big[ \I^{1-\a}_{a+} [q] \Big] (t) , $$
for almost every $t \in I$. In particular $\Dm [q] \in \L^1_\loc(I,\R)$.
\end{definition}

\begin{definition}
Let $0 \leq \a \leq 1$. We denote by $\AC^\alpha_{a+}(I,\R)$ the set of all functions $q \in \L^1_\loc (I,\R)$ possessing on $I$ a left R-L fractional derivative $\Dm [q]$ of order $\a$.
\end{definition}

\noindent If $\a = 1$, $\AC^1_{a+}(I,\R) = \AC_\loc (I,\R)$ and $\D^1_{a+} [q] = \dot{q}$ for any $q \in \AC_\loc (I,\R)$.

\noindent If $\a = 0$, $\AC^0_{a+}(I,\R) = \L^1_\loc (I,\R)$ and $\D^0_{a+} [q] = q$ for any $q \in \L^1_\loc (I,\R)$.

\begin{definition}
We say that $q \in \C(I,\R)$ possesses on $I$ a left Caputo fractional derivative $\CDm [q]$ of order $0 \leq \a \leq 1$ if and only if $q-q(a) \in \AC^\a_{a+} (I,\R)$. In that case $\CDm [q]$ is defined by
$$ \CDm [q](t) := \Dm [q-q(a)] (t) , $$
for almost every $t \in I$. In particular $\CDm [q] \in \L^1_\loc(I,\R)$.
\end{definition}

\begin{definition}
Let $0 \leq \a \leq 1$. We denote by ${}_\c \AC^\alpha_{a+}(I,\R)$ the set of all functions $q \in \C(I,\R)$ possessing on $I$ a left Caputo fractional derivative $\CDm [q]$ of order $\a$.
\end{definition}

\noindent If $\a = 1$, ${}_\c \AC^1_{a+}(I,\R) = \AC_\loc (I,\R)$ and  ${}_\c \D^1_{a+} [q] = \dot{q}$ for any $q \in \AC_\loc (I,\R)$.

\noindent If $\a = 0$, ${}_\c \AC^0_{a+}(I,\R) = \C(I,\R)$ and ${}_\c \D^0_{a+} [q] = q-q(a)$ for any $q \in \C(I,\R)$.

\begin{example}\label{ex1}
The constant function $q = 1 \in \AC^\alpha_{a+}(I,\R) \cap {}_\c \AC^\alpha_{a+}(I,\R)$ for any $0 \leq \a \leq 1$. It holds that $\CDm [1] = 0$ for any $0 \leq \a \leq 1$. We also have $\D^1_{a+}[1] =0$ and
$$ \D^\a_{a+} [1] (t) = \dfrac{1}{\Gamma (1-\a)} (t-a)^{-\a}, $$
for any $0 \leq \a < 1$ and for every $ t \in I$, $t > a$. 
\end{example}

\subsection{Some preliminaries on matrix computations}\label{secbasics2}
In this section we fix $m$, $n$, $k \in \N^*$. For any couple of matrices $A = (A_{ij}) \in \R^{m \times n}$, $B = (B_{ij}) \in \R^{n \times k}$, we denote by $A \times B \in \R^{m \times k}$ the usual matrix-matrix product. The notation $\times$ will also be used for the classical matrix-vector product (\textit{i.e.} for $k=1$). 

\smallskip

For any couple of same size matrices $A = (A_{ij})$, $B = (B_{ij}) \in \R^{m \times n}$, we denote by $A \otimes B$ the classical Hadamard product given by
$$ A \otimes B := \left( 
\begin{array}{ccc}
A_{11} B_{11} & \cdots & A_{1n} B_{1n} \\
\vdots & \ddots & \vdots \\
A_{m1} B_{m1} & \cdots & A_{mn} B_{mn}
\end{array}
\right) \in \R^{m \times n}. $$

For a vector $A := (A_i) \in \R^{m,1}$, we denote by $\overline{A} \in \R^{m \times m}$ the row constant square matrix given by
$$ \overline{A} := \left( 
\begin{array}{cccc}
A_{1} & A_1 & \cdots & A_{1} \\
A_{2} & A_2 & \cdots & A_{2} \\
\vdots & \vdots & \ddots & \vdots \\
A_{m} & A_{m} & \cdots & A_{m}
\end{array}
\right) , $$
and by $\underline{A} \in \R^{m \times m}$ the column constant square matrix given by $ \underline{A} := \overline{A}^\top$, where $\overline{A}^\top$ denotes the transpose of $\overline{A}$. One can easily prove the following series of lemmas. They will be useful in particular in Appendix~\ref{appsectransition}. 

\begin{lemma}\label{lemcalcul4}
Let $A = (A_{i}) \in \R^{m \times 1}$ be a vector. Then 
$$ \overline{A} \otimes \ID_m = \underline{A} \otimes \ID_m .$$
\end{lemma}

\begin{lemma}\label{lemcalcul}
Let $A = (A_{i}) \in \R^{m \times 1}$ be a vector and $B = (B_{ij}) \in \R^{m \times m}$ be a square matrix. Then 
$$ (\overline{A} \otimes B) \times X = A \otimes (B \times X) , $$
for any vector $X = (X_{i}) \in \R^{m \times 1}$.
\end{lemma}

\begin{lemma}\label{lemcalcul2}
Let $A = (A_{i}) \in \R^{m \times 1}$, $B = (B_{i}) \in \R^{m \times 1}$ be two vectors and $C = (C_{ij}) \in \R^{m \times m}$ be a square matrix. Then
$$ \underline{A} \otimes \Big[ ( \overline{B} \otimes \ID_m) \times C \Big] = \overline{B} \otimes \Big[ C \times (\underline{A} \otimes \ID_m ) \Big] .$$
\end{lemma}

\begin{lemma}\label{lemcalcul3}
Let $A = (A_{i}) \in \R^{m \times 1}$, $B = (B_{i}) \in \R^{m \times 1}$ be two vectors and $C = (C_{ij}) \in \R^{m \times m}$, $D = (D_{ij}) \in \R^{m \times m}$ and $E = (E_{ij}) \in \R^{m \times m}$ be three square matrices. Then
$$ \underline{A} \otimes \Big[ \Big( \overline{B} \otimes ( C \times D ) \Big) \times E \Big] = \overline{B} \otimes \Big[ C \times \Big( \underline{A} \otimes ( D \times E ) \Big) \Big] .$$
\end{lemma}

\subsection{Multi-order fractional calculus for matrix functions}\label{secmultiorder}

In the whole section we fix $a \in \R$, $I \in \IIm$ and $m$, $n \in \N^*$. 

\begin{definition}
The left R-L fractional integral $\Im [A]$ of multi-order $\a = (\a_{ij}) \in (\R^+)^{m \times n}$ of a matrix function $A = (A_{ij}) \in \L^1_\loc (I,\R^{m \times n})$ is defined by
$$ \Im [A](t) := \left(
\begin{array}{ccc}
\I^{\alpha_{11}}_{a+} [A_{11}](t) & \cdots & \I^{\alpha_{1n}}_{a+} [A_{1n}](t) \\
\vdots & \ddots & \vdots \\
\I^{\alpha_{m1}}_{a+} [A_{m1}](t) & \cdots & \I^{\alpha_{mn}}_{a+} [A_{mn}](t)
\end{array}
\right) , $$
for almost every $t \in I$.
\end{definition}

\noindent For the ease of notations, we introduce
$$ \left[ \dfrac{1}{\Gamma(\a)} (t-\t)^{\a-1} \right] := \left( \dfrac{1}{\Gamma(\a_{ij})} (t-\t)^{\a_{ij}-1} \right)_{ij} \in \R^{m \times n} , $$
and we write
$$ \Im [A](t) = \int_a^t \left[ \dfrac{1}{\Gamma(\a)} (t-\t)^{\a-1} \right] \otimes A(\t) \; d\t, $$
for every $\a = (\a_{ij}) \in (\R^+_*)^{m \times n}$ and $A = (A_{ij}) \in \L^1_\loc (I,\R^{m \times n})$.

\bigskip

Similarly to Section~\ref{secbasics1}, one can easily define the corresponding operators $\Dm$, $\CDm$ and the corresponding sets $\AC^\a_{a+}(I,\R^{m \times n})$, ${}_\c \AC^\a_{a+}(I,\R^{m \times n})$ for any matrix fractional multi-order $\a = (\a_{ij}) \in [0,1]^{m \times n}$. All statements of Section~\ref{secbasics1} can be extended to matrix functions and to matrix fractional multi-orders. 

\section{Two non-linear fractional multi-order vector Cauchy problems}\label{seccauchyproblems}
In the whole section we are interested in non-linear fractional multi-order Cauchy problems. Since the dynamics are non-linear, it is not of interest to consider matrix Cauchy problems. Indeed, $\R^{m \times n}$ can be identified to $\R^{mn}$ and it is sufficient to consider vector Cauchy problems. As a consequence, we fix in this section $m \in \N^*$ and $n=1$. The notation $\vert \cdot \vert_m$ stands for the Euclidean norm of $\R^m$ and $\BB_m(x,R)$ stands for the closed ball of $\R^m$ centered at $x \in \R^m$ and with radius $R > 0$.

\bigskip

Let $a \in \R$ and let $f:\Omega \times I_f \longrightarrow \R^m$, $(x,t) \longmapsto f(x,t)$ be a Carath\'eodory function, where $I_f \in \IIm$ and $\Omega$ is a nonempty open subset of $\R^m$. Finally, let $q_a \in \Omega$ and let $\alpha = (\a_1,\ldots,\a_m) \in (0,1]^m$ be a vector fractional multi-order. In this section we are interested in two different non-linear fractional multi-order vector Cauchy problems.
\begin{itemize}
\item The first vector Cauchy problem~\eqref{eqcauchyproblemRL} is given by
\begin{equation}\label{eqcauchyproblemRL}\tag{VCP}
\left\lbrace \begin{array}{l}
\Dm [q](t) = f(q(t),t),  \\
\I^{1-\a}_{a+} [q](a) = q_a, 
\end{array}
\right.
\end{equation}
that involves a R-L fractional derivative $\Dm$ and the initial condition $\I^{1-\a}_{a+} [q](a) = q_a$. We will study this problem in Section~\ref{sectionCauchyRL}, only in the case $\Omega = \R^m$. 
\item The second vector Cauchy problem~\eqref{eqcauchyproblemC} is given by 
\begin{equation}\label{eqcauchyproblemC}\tag{${}_\c$VCP}
\left\lbrace \begin{array}{l}
\CDm [q](t) = f(q(t),t),  \\
q(a) = q_a, 
\end{array}
\right.
\end{equation}
that involves a Caputo fractional derivative $\CDm$ and the initial condition $q(a) = q_a$. We will study this problem in Section~\ref{sectionCauchyC}. In Section~\ref{sectionCauchyC}, in contrary to Section~\ref{sectionCauchyRL}, we will not restrict $\Omega$ to be the entire space $\R^m$.
\end{itemize}

\subsection{An existence-uniqueness result for \eqref{eqcauchyproblemRL}}\label{sectionCauchyRL}

In the whole section we assume that $\Omega = \R^m$. All proofs of this section are detailed in Appendix~\ref{appsectionCauchyRL}.

\subsubsection{Properties of the dynamic $f$}
As in the classical Cauchy-Lipschitz theory, the existence and uniqueness of a solution of~\eqref{eqcauchyproblemRL} require some assumptions on the dynamic $f$, whence the following series of definitions. 

\begin{definition}
The dynamic $f$ is said to be \emph{preserving the integrability of zero} if 
\begin{equation}\label{HL1zero}\tag{Hyp${}^0_{1}$}
f(0,\cdot) \in \L^1_\loc(I_f,\R^m). 
\end{equation}
In what follows this property will be referred to as~\eqref{HL1zero}.
\end{definition}

\begin{definition}
The dynamic $f$ is said to be \emph{preserving the integrability} if 
\begin{equation}\label{HL1}\tag{Hyp${}_{1}$}
f(q,\cdot) \in \L^1_\loc(I_f,\R^m) , 
\end{equation}
for any $q \in \L^1_\loc(I_f,\R^m)$. In what follows this property will be referred to as~\eqref{HL1}.
\end{definition}

\begin{definition}
The dynamic $f$ is said to be \emph{globally Lipschitz continuous in its first variable} if for any $[c,d] \subset I_f$, there exists $L \geq 0$ such that
\begin{equation}\label{HgloblipRL}\tag{Hyp$_{\mathrm{glob}}$}
\vert f(x_2,t) - f(x_1,t) \vert_m \leq L \vert x_2 - x_1 \vert_m,
\end{equation}
for any $x_1$, $x_2\in \R^m$ and for almost every $t \in [c,d]$. In what follows this property will be referred to as~\eqref{HgloblipRL}.
\end{definition}

\noindent Note that if $f$ satisfies~\eqref{HgloblipRL}, then $f$ satisfies~\eqref{HL1} if and only if $f$ satisfies \eqref{HL1zero}.

\subsubsection{Definition of a global solution and main results}
We introduce here a notion of (global) solution of~\eqref{eqcauchyproblemRL}.

\begin{definition}
A function $q : I_f \to \R^m$ is said to be a \textit{(global) solution} of~\eqref{eqcauchyproblemRL} if and only if
\begin{itemize}
\item $q \in \AC^\a_{a+}(I_f,\R^m)$;
\item $\I^{1-\a}_{a+}[q](a) = q_a$;
\item $\Dm [q](t) = f(q(t),t)$ for almost every $t \in I_f$.
\end{itemize}
\end{definition}
\smallskip
\noindent The following proposition gives an integral representation for (global) solutions of~\eqref{eqcauchyproblemRL}.

\begin{proposition}[Integral representation]\label{prop1RL}
If $f$ satisfies~\eqref{HL1}, a function $q : I_f \to \R^m$ is a \textit{(global) solution} of~\eqref{eqcauchyproblemRL} if and only if $q \in \L^1_\loc(I_f,\R^m)$ and
\begin{eqnarray*}
q(t) & = & \D^{1-\alpha}_{a+}[ q_a ] (t) + \Im [ f(q,\cdot) ](t) , \\[4pt]
& = & \left[ \dfrac{1}{\Gamma (\a) } (t-a)^{\a-1} \right] \otimes q_a + \int_a^t \left[ \dfrac{1}{\Gamma(\a)} (t-\t)^{\a-1} \right] \otimes f(q(\t),\t) \; d\t,
\end{eqnarray*}
for almost every $t \in I_f$.
\end{proposition}

\noindent The next theorem provides an existence-uniqueness result for~\eqref{eqcauchyproblemRL}.

\begin{theorem}\label{thm1RL}
If $f$ satisfies \eqref{HL1zero} and \eqref{HgloblipRL}, then \eqref{eqcauchyproblemRL} has a unique (global) solution.
\end{theorem}

\noindent Similar results were already obtained in the literature. We refer to Introduction for details and references. Note that the proof of Theorem~\ref{thm1RL}, detailed in Appendix~\ref{appsectionCauchyRL}, is based on the introduction of an appropriate Bielecki norm. This method is widely inspired from \cite{idcz}.

\subsection{Existence-uniqueness results for \eqref{eqcauchyproblemC}}\label{sectionCauchyC}

In the whole section we consider that $\Omega$ is a nonempty open subset of $\R^m$. In the sequel $\KK$ stands for the set of compact subsets of $\Omega$. All results of this section are detailed in Appendix~\ref{appsectionCauchyC}.

\subsubsection{Properties of the dynamic $f$}
As in the classical Cauchy-Lipschitz theory, the existence and uniqueness of a solution of~\eqref{eqcauchyproblemC} require some assumptions on the dynamic $f$, whence the following series of definitions. 

\begin{definition}
The dynamic $f$ is said to be \emph{bounded on compacts} if, for any $K \in \KK$ and for any $[c,d] \subset I_f$, there exists $M \geq 0$ such that
\begin{equation}\label{Hinfty}\tag{Hyp${}_\infty$}
\vert f(x,t) \vert_m \leq M,
\end{equation}
for any $x\in K$ and for almost every $t \in [c,d]$. In what follows this property will be referred to as~\eqref{Hinfty}.
\end{definition}

\begin{definition}
The dynamic $f$ is said to be \emph{locally Lipschitz continuous in its first variable} if, for every $(x,t) \in \Omega \times I_f$, there exist $R > 0$, $\delta > 0$ and $L \geq 0$ such that $\BB_m(x,R) \subset \Omega$ and
\begin{equation}\label{Hloclip}\tag{Hyp${}_{\mathrm{loc}}$}
\vert f(x_2,\t) - f(x_1,\t) \vert_m \leq L \vert x_2 - x_1 \vert_m,
\end{equation}
for any $ x_1$, $x_2 \in \BB_m(x,R)$ and for almost every $\t \in [t-\delta,t+\delta] \cap I_f$. In what follows this property will be referred to as~\eqref{Hloclip}.
\end{definition}

\begin{definition}
The dynamic $f$ is said to be \emph{globally Lipschitz continuous in its first variable} if for any $[c,d] \subset I_f$, there exists $L \geq 0$ such that
\begin{equation}\label{HgloblipC}\tag{Hyp$_{\mathrm{glob}}$}
\vert f(x_2,t) - f(x_1,t) \vert_m \leq L \vert x_2 - x_1 \vert_m,
\end{equation}
for any $x_1$, $x_2\in\Omega$ and for almost every $t \in [c,d]$. In what follows this property will be referred to as~\eqref{HgloblipC}.
\end{definition}

\noindent Note that if $f$ satisfies~\eqref{HgloblipC}, then $f$ satisfies~\eqref{Hloclip}. 

\noindent Note that if $f$ satisfies~\eqref{Hinfty}, then $f$ satisfies~\eqref{HL1zero}.

\subsubsection{Definition of a maximal solution and main results}

We introduce
$$ \IImf := \{ I \in \IIm \text{ such that } I \subset I_f \}. $$
Now we introduce a notion of local solution of~\eqref{eqcauchyproblemC}.

\begin{definition}
A couple $(q,I)$ is said to be a \textit{local solution} of~\eqref{eqcauchyproblemC} if and only if
\begin{itemize}
\item $I \in \IImf$;
\item $q \in {}_\c \AC^\a_{a+}(I,\Omega)$ (in particular $q$ is with values in $\Omega$);
\item $q(a) = q_a$;
\item $\CDm [q](t) = f(q(t),t)$ for almost every $t \in I$.
\end{itemize}
\end{definition}

\begin{definition}
Let $(q,I)$ be a local solution of~\eqref{eqcauchyproblemC}. We say that $(q',I')$ is an \textit{extension} of $(q,I)$ if $(q',I')$ is a local solution of~\eqref{eqcauchyproblemC} and if $I \subset I'$ and $q' = q$ on $I$.
\end{definition}

\begin{definition}
Let $(q,I)$ be a local solution of~\eqref{eqcauchyproblemC}. We say that $(q,I)$ is a \textit{maximal solution} of~\eqref{eqcauchyproblemC} if $I' = I$ for any extension $(q',I')$ of $(q,I)$. 
\end{definition}

\begin{definition}
Let $(q,I)$ be a local solution of~\eqref{eqcauchyproblemC}. We say that $(q,I)$ is a \textit{global solution} of~\eqref{eqcauchyproblemC} if $I = I_f$. 
\end{definition}

\noindent Note that a global solution of~\eqref{eqcauchyproblemC} is necessarily maximal. The following proposition gives an integral representation for local solutions of~\eqref{eqcauchyproblemC}.

\begin{proposition}[Integral representation]\label{prop1C}
If $f$ satisfies~\eqref{Hinfty}, a couple $(q,I)$ is a \textit{local solution} of~\eqref{eqcauchyproblemC} if and only if $I \in \IImf$, $q \in \C(I,\Omega)$ and
\begin{eqnarray*}
q(t) & = & q_a + \Im [ f(q,\cdot) ](t) , \\[4pt]
& = & q_a + \int_a^t \left[ \dfrac{1}{\Gamma(\a)} (t-\t)^{\a-1} \right] \otimes f(q(\t),\t) \; d\t,
\end{eqnarray*}
for every $t \in I$.
\end{proposition}

\noindent The next theorems provide existence-uniqueness results for~\eqref{eqcauchyproblemC}.

\begin{theorem}\label{thm1}
If $f$ satisfies \eqref{Hinfty} and \eqref{Hloclip}, then~\eqref{eqcauchyproblemC} has a unique maximal solution $(q,I)$. Moreover $(q,I)$ is the maximal extension of any other local solution of~\eqref{eqcauchyproblemC}.
\end{theorem}

\begin{theorem}\label{thm3}
If $\; \Omega = \R^m$ and if $f$ satisfies \eqref{Hinfty} and \eqref{HgloblipC}, then the maximal solution $(q,I)$ of~\eqref{eqcauchyproblemC} is global, that is, $I = I_f$.
\end{theorem}

\noindent Similar results were already obtained in the literature. We refer to Introduction for details and references. 

\begin{remark}
If $\Omega = \R^m$, if $f$ satisfies \eqref{Hinfty} and \eqref{HgloblipC} and if $q_a = 0$, then the unique maximal solution (that is moreover global) of~\eqref{eqcauchyproblemC} coincides with the unique global solution of~\eqref{eqcauchyproblemRL}. In particular, in that case, the unique global solution of~\eqref{eqcauchyproblemRL} belongs to $\C_a (I_f,\R^m)$.
\end{remark}

\noindent As far as we know, the following last result was not addressed in the literature yet. It provides informations on the behavior of a maximal solution. Precisely, it states that a maximal solution that is not global must go out of any compact of $\Omega$.

\begin{theorem}\label{thm2}
If $f$ satisfies \eqref{Hinfty} and \eqref{Hloclip} and if $(q,I)$ is the maximal solution of~\eqref{eqcauchyproblemC}, then:
\begin{itemize}
\item either $I = I_f$, that is, $(q,I)$ is global;
\item either $I = [a,b)$ with $b \in I_f$, $b > a$, and moreover, for every $K \in \KK$, there exists $t\in I$ such that $q(t) \notin K$.
\end{itemize}
\end{theorem}

\section{Fractional state-transition matrices}\label{sectransition}

In Section~\ref{sectransition1} we focus on homogeneous linear square matrix Cauchy problems and we define fractional state-transition matrices. Our aim is to provide in Section~\ref{secduhamel} fractional versions of the classical Duhamel formula. Finally, Sections~\ref{sectransition3} and \ref{sectransition4} are devoted to duality theorems relying left and right state-transition matrices. All proofs of Section~\ref{sectransition} are detailed in Appendix~\ref{appsectransition}.

\subsection{Definitions}\label{sectransition1}
In the whole section we fix $a \in \R$, $I \in \IIm$ and $m \in \N^*$. Let us consider a square matrix function $A =(A_{ij} ) \in \L^\infty_\loc(I,\R^{m \times m})$ and a square matrix fractional multi-order $\a = (\a_{ij}) \in (0,1]^{m \times m}$. For every $s \in I$, $s < \sup I$, we denote by $I^s := I \cap [s,+\infty)$. Note that $I^s \in \mathbb{I}_{s+}$. The following Proposition-Definitions clearly follow from Propositions~\ref{prop1RL} and \ref{prop1C} and from Theorems~\ref{thm1RL}, \ref{thm1} and \ref{thm3}.

\begin{definition}\label{defRLstate}
For every $s \in I$, $s < \sup I$, the homogeneous linear square matrix Cauchy problem given by
\begin{equation}\label{eqcauchyproblemRLLM}\tag{LMCP}
\left\lbrace \begin{array}{l}
\D^\a_{s+} [Z](t) = A(t) \times Z(t),  \\
\I^{1-\a}_{s+} [Z](s) = \ID_m, 
\end{array}
\right.
\end{equation}
admits, in virtue of Theorem~\ref{thm1RL}, a unique (global) solution denoted by $Z(\cdot,s) \in \AC^{\a}_{a+}(I^s,\R^{m \times m})$.
The function $Z(\cdot,\cdot)$ is called the \textit{left R-L state-transition matrix} associated to $A$ and $\a$. It follows from Proposition~\ref{prop1RL} that
$$ Z(t,s) = \left[ \dfrac{1}{\Gamma(\a)} (t-s)^{\a-1} \right] \otimes \ID_m + \int_s^t \left[ \dfrac{1}{\Gamma(\a)} (t-\t)^{\a-1} \right] \otimes \Big[ A(\t) \times Z(\t,s) \Big] \; d\t, $$
for almost every $t$, $s \in I$ with $t > s$. 
\end{definition}

\begin{definition}\label{defCstate}
For every $s \in I$, $s < \sup I$, the homogeneous linear square matrix Cauchy problem given by
\begin{equation}\label{eqcauchyproblemCLM}\tag{${}_\c$LMCP}
\left\lbrace \begin{array}{l}
\CD^\a_{s+} [Z](t) = A(t) \times Z(t),  \\
Z(s) = \ID_m, 
\end{array}
\right.
\end{equation}
admits, in virtue of Theorems~\ref{thm1} and \ref{thm3}, a unique maximal solution, that is moreover global, denoted by ${}_\c Z(\cdot,s) \in {}_\c \AC^{\a}_{a+}(I^s,\R^{m \times m})$. The function ${}_\c Z(\cdot,\cdot)$ is called the \textit{left Caputo state-transition matrix} associated to $A$ and $\a$. It follows from Proposition~\ref{prop1C} that
$$ {}_\c Z(t,s) = \ID_m + \int_s^t \left[ \dfrac{1}{\Gamma(\a)} (t-\t)^{\a-1} \right] \otimes \Big[ A(\t) \times {}_\c Z(\t,s) \Big] \; d\t, $$
for every $t$, $s \in I$ with $t \geq s$. 
\end{definition}

\begin{example}
As recalled and referenced in Introduction, if $A(\cdot) = A $ is constant and if $\a$ is row and column constant, then
$$ Z(t,s) = (t-s)^{\a-1} E_{\a,\a}  (A(t-s)^\a)  \quad \text{and} \quad {}_\c Z(t,s) = E_{\a,1} (A(t-s)^\a) , $$
where $E_{\a,\beta}$ denotes the classical Mittag-Leffler function. We refer to \cite{das,idcz} for more details.
\end{example}

\noindent We are now in a position to state fractional versions of the classical Duhamel formula in the next section. Before coming to that point, we first need to state the following technical but useful lemma.

\begin{lemma}\label{lemimportant}
Let $b \in I$ with $b > a$. There exists $\Theta^b \geq 0$ such that
$$ \vert Z_{ij} (t,s) \vert \leq (t-s)^{\a_{ij}-1} \Theta^b , $$
for almost every $a \leq s < t \leq b$ and for every $i$, $j \in \{ 1,\ldots,m \}$. In particular, $Z(t,\cdot) \in \L^1([a,t],\R^{m \times m})$ for almost every $t \in I$, $t>a$.
\end{lemma}

\subsection{Fractional Duhamel formulas}\label{secduhamel}
In this section we fix $a \in \R$, $I \in \IIm$ and $m \in \N^*$. Let $q_a \in \R^m$ and let $\a = (\a_{i}) \in (0,1]^{m}$ be a vector fractional multi-order. Let us consider a square matrix function $A = (A_{ij}) \in \L^\infty_\loc(I,\R^{m \times m})$ and a vector function $B = (B_i) \in \L^\infty_\loc(I,\R^m)$. 

\smallskip

Let $Z(\cdot,\cdot)$ be the left R-L state-transition matrix associated to $A$ and $\aa \in (0,1]^{m \times m}$. Let ${}_\c Z(\cdot,\cdot)$ be the left Caputo state-transition matrix associated to $A$ and $\aa \in (0,1]^{m \times m}$. The main results of this paper are stated as follows.

\begin{theorem}[Duhamel formula]\label{thmduhamelRL}
The non-homogeneous linear vector Cauchy problem given by
\begin{equation}\label{eqcauchyproblemRLLV}\tag{LVCP}
\left\lbrace \begin{array}{l}
\Dm [q](t) = A(t) \times q(t) + B(t),  \\
\I^{1-\a}_{a+}[q](a) = q_a, 
\end{array}
\right.
\end{equation}
admits a unique (global) solution denoted by $q$ and it is given by the fractional Duhamel formula
$$ q(t) = Z(t,a) \times q_a + \int_a^t Z(t,s) \times B(s) \; ds , $$
for almost every $t \in I$.
\end{theorem}

\begin{theorem}[Duhamel formula]\label{thmduhamelC}
The non-homogeneous linear vector Cauchy problem given by
\begin{equation}\label{eqcauchyproblemCLV}\tag{${}_\c$LVCP}
\left\lbrace \begin{array}{l}
\CDm [q](t) = A(t) \times q(t) + B(t),  \\
q(a) = q_a, 
\end{array}
\right.
\end{equation}
admits a unique maximal solution, that is moreover global, denoted by $q$ and it is given by the fractional Duhamel formula
$$ q(t) = {}_\c Z(t,a) \times q_a + \int_a^t Z(t,s) \times B(s) \; ds , $$
for every $t \in I$.
\end{theorem}

\noindent In the fractional Duhamel formula associated to~\eqref{eqcauchyproblemCLV}, note that both $Z(\cdot,\cdot)$ and ${}_\c Z(\cdot,\cdot)$ are involved.

\begin{remark}
By curiosity one would wonder what are the Cauchy problems associated to the functions $q_1$, $q_2$ defined by
$$ q_1 (t) := Z(t,a) \times q_a + \int_a^t {}_\c Z(t,s) \times B(s) \; ds  \; \; \; \; \text{and} \; \; \; \;  q_2 (t) := {}_\c Z(t,a) \times q_a + \int_a^t {}_\c Z(t,s) \times B(s)  \; ds  . $$
Similarly to the proofs of Theorems~\ref{thmduhamelRL} and \ref{thmduhamelC}, it can be proved that $q_1$ is the unique global solution of 
\begin{equation*}
\left\lbrace \begin{array}{l}
\Dm [q](t) = A(t) \times q(t) + \I^{1-\a}_{a+}[B](t),  \\
\I^{1-\a}_{a+}[q](a) = q_a, 
\end{array}
\right.
\end{equation*}
and $q_2$ is the unique global solution of
\begin{equation*}
\left\lbrace \begin{array}{l}
\CDm [q](t) = A(t) \times q(t) + \I^{1-\a}_{a+}[B](t),  \\
q(a) = q_a.
\end{array}
\right.
\end{equation*}
\end{remark}

\subsection{Preliminaries and recalls on right fractional operators}\label{sectransition3}
In Section~\ref{secbasics1} we have recalled the usual definitions and results about left fractional operators. The corresponding right fractional operators are defined as follows. We fix $b \in \R$ and $I \in \IIp$ where 
$$ \IIp := \{ I \subset \R \text{ interval such that } \{ b \} \varsubsetneq I \subset (-\infty,b] \}. $$

\begin{definition}
The right R-L fractional integral $\Ip [q]$ of order $\a > 0$ of $q \in \L^1_\loc (I,\R)$ is defined on $I$ by
$$ \Ip [q](t) :=\int_t^b \dfrac{1}{\Gamma(\a)}  (\t-t)^{\a-1} q(\t) \; d\t, $$
provided that the right-hand side term exists. For $\alpha = 0$ and $q \in \L^1_\loc (I,\R)$, we define $\I^0_{b-}[q] := q$.
\end{definition}

\begin{definition}
We say that $q \in \L^1_\loc (I,\R)$ possesses on $I$ a right R-L fractional derivative $\Dp [q]$ of order $0 \leq \a \leq 1$ if and only if $\I^{1-\a}_{b-} [q] \in \AC_\loc (I,\R)$. In that case $\Dp [q]$ is defined by
$$ \Dp [q](t) := - \dfrac{d}{dt} \Big[ \I^{1-\a}_{b-} [q] \Big] (t) , $$
for almost every $t \in I$. In particular $\Dp [q] \in \L^1_\loc(I,\R)$.
\end{definition}

\begin{definition}
Let $0 \leq \a \leq 1$. We denote by $\AC^\alpha_{b-}(I,\R)$ the set of all functions $q \in \L^1_\loc (I,\R)$ possessing on $I$ a right R-L fractional derivative $\Dp [q]$ of order $\a$.
\end{definition}

\begin{definition}
We say that $q \in \C(I,\R)$ possesses on $I$ a right Caputo fractional derivative $\CDp [q]$ of order $0 \leq \a \leq 1$ if and only if $q-q(b) \in \AC^\a_{b-} (I,\R)$. In that case $\CDp [q]$ is defined by
$$ \CDp [q](t) := \Dp [q-q(b)] (t) , $$
for almost every $t \in I$. In particular $\CDp [q] \in \L^1_\loc(I,\R)$.
\end{definition}

\begin{definition}
Let $0 \leq \a \leq 1$. We denote by ${}_\c \AC^\alpha_{b-}(I,\R)$ the set of all functions $q \in \C(I,\R)$ possessing on $I$ a right Caputo fractional derivative $\CDp [q]$ of order $\a$.
\end{definition}

\noindent All results of Section~\ref{secbasics1} can be extended to right fractional operators. Similarly to Section~\ref{secmultiorder}, the right fractional operators can be extended to matrix functions and to matrix fractional multi-orders. Finally, all results about left Cauchy problems obtained in Sections~\ref{seccauchyproblems} and \ref{sectransition} can also be adapted to the right case.

\subsection{Duality theorems}\label{sectransition4}

In this section we fix $a \in \R$, $I \in \IIm$ and $m \in \N^*$. Let $A = (A_{ij}) \in \L^\infty_\loc(I,\R^{m \times m})$ be a square matrix function and let $\a = (\a_{i}) \in (0,1]^{m}$ be a vector fractional multi-order. 

\smallskip

The following duality theorem states that the left R-L state-transition matrix associated to $A$ and $\aa \in (0,1]^{m \times m}$ coincides with the right R-L state transition matrix associated to $A$ and $\aaa \in (0,1]^{m \times m}$.

\begin{theorem}[Duality theorem]\label{thmduality}
Let $Z(\cdot,\cdot)$ be the left R-L state-transition matrix associated to $A$ and $\aa \in (0,1]^{m \times m}$. Then, $Z(t,\cdot)$ is the unique (global) solution of
\begin{equation*}
\left\lbrace \begin{array}{l}
\D^{\aaa}_{t-} [Z](s) = Z(s) \times A(s) ,  \\
\I^{1-\aaa}_{t-} [Z](t) = \ID_m,
\end{array}
\right.
\end{equation*}
for almost every $t \in I$, $t > a$.
\end{theorem}

The exact analogous of the above theorem for the left Caputo state-transition matrix does not hold true in general. Indeed, one can easily see that the proof of Theorem~\ref{thmduality} cannot be adapted to this case. Nevertheless, the following duality theorem can be proved if $A(\cdot) = A$ is constant.

\begin{theorem}[Duality theorem]\label{thmduality2}
Let us assume that $A(\cdot) = A$ is constant and let ${}_\c Z(\cdot,\cdot)$ be the left Caputo state-transition matrix associated to $A$ and $\aa \in (0,1]^{m \times m}$. Then, ${}_\c Z(t,\cdot)$ is the unique maximal solution, that is moreover global, of
\begin{equation*}
\left\lbrace \begin{array}{l}
\CD^{\aaa}_{t-} [Z](s) = Z(s) \times A(s) ,  \\
Z(t) = \ID_m,
\end{array}
\right.
\end{equation*}
for every $t \in I$, $t > a$.
\end{theorem}

\appendix

\section{Proofs of Section~\ref{sectionCauchyRL}}\label{appsectionCauchyRL}

\subsection{Proof of Proposition~\ref{prop1RL}}

We first prove the necessary condition. Let $q : I_f \to \R^m$ be a (global) solution of~\eqref{eqcauchyproblemRL}. Since $\I^{1-\a}_{a+} [q] \in \AC_\loc(I_f,\R^m)$ and $\I^{1-\a}_{a+} [q](a) = q_a$, it holds that
$$ \I^{1-\a}_{a+} [q] = \I^{1-\a}_{a+} [q](a) + \I^1_{a+} \left[ \frac{d}{dt} \Big[ \I^{1-\a}_{a+} [q] \Big] \right] = q_a + \I^1_{a+} \Big[ \Dm [q] \Big] = q_a + \I^1_{a+} [ f(q,\cdot) ] , $$
everywhere on $I_f$. Since $q$, $q_a$ and $\Dm [q] = f(q,\cdot) \in \L^1_\loc (I_f,\R^m)$, it holds from Proposition~\ref{prop5} that
$$ \I^1_{a+} [q] = \Im [q_a] + \I^1_{a+} \Big[ \Im [f(q,\cdot)] \Big] ,$$
almost everywhere on $I_f$, and then everywhere on $I_f$ from continuity. Since $q$ and $ \Im [f(q,\cdot)]  \in \L^1_\loc (I_f,\R^m)$, differentiating the previous equality leads to
$$ q(t) = \D^{1-\a}_{a+} [q_a] (t) + \Im [f(q,\cdot] (t) ,$$
for almost every $t \in I_f$.\footnote{Note that Hypothesis~\eqref{HL1} is not required for the necessary condition.} 

\smallskip

Now let us prove the sufficient condition. Since $q \in \L^1_\loc (I_f,\R^m)$ and since $f$ satisfies~\eqref{HL1}, it holds that $f(q,\cdot) \in \L^1_\loc(I_f,\R^m)$. We also know that $\D^{1-\a}_{a+} [q_a] \in \L^1_\loc(I_f,\R^m)$ and one can easily prove from the classical Beta function that $\I^{1-\a}_{a+} [ \D^{1-\a}_{a+} [q_a] ] = q_a$ everywhere on $I_f$. Finally, since we have
$$ q = \D^{1-\a}_{a+} [q_a] + \Im [f(q,\cdot)], $$
almost everywhere on $I_f$, we get from Proposition~\ref{prop5} that
$$ \I^{1-\a}_{a+} [q] = q_a + \I^1_{a+} [ f(q,\cdot)] , $$ 
almost everywhere on $I_f$. Since $f(q,\cdot) \in \L^1_\loc(I_f,\R^m)$, we have $\I^1_{a+} [ f(q,\cdot)] \in \AC_{a,\loc}(I_f,\R^m)$. Then $\I^{1-\a}_{a+}[q]$ can be identified to $q_a + \I^1_{a+} [ f(q,\cdot)] \in \AC_\loc(I_f,\R^m)$. Thus $q \in \AC^\a_{a+}(I_f,\R^m)$ with $\I^{1-\a}_{a+}[q](a) = q_a$ and 
$$ \Dm [q] = \dfrac{d}{dt} \Big[  \I^{1-\a}_{a+} [q]  \Big] = f(q,\cdot) , $$
almost everywhere on $I_f$.

\subsection{Preliminary lemmas for Theorem~\ref{thm1RL}}
We introduce
$$ \IImf := \{ I \in \IIm \text{ such that } I \subset I_f \}. $$
In order to prove Theorem~\ref{thm1RL} in the next section, we first prove in this section two preliminary lemmas.

\begin{lemma}\label{lem1}
Let $I \in \IImf$. If $q : I_f \to \R^m$ is a \textit{(global) solution} of~\eqref{eqcauchyproblemRL}, then the restriction $q_{|I} : I \to \R^m$ is a \textit{(global) solution} of the restricted Cauchy problem~\eqref{eqcauchyproblemRLrestricted} given by
\begin{equation}\label{eqcauchyproblemRLrestricted}\tag{VCP${}_{|I}$}
\left\lbrace \begin{array}{l}
\Dm [q](t) = f_{|I}(q(t),t),  \\
\I^{1-\a}_{a+} [q](a) = q_a,
\end{array}
\right.
\end{equation}
where $f_{|I}$ is the restriction $f_{|I} : \R^m \times I \to \R^m$ of $f$.
\end{lemma}

\begin{proof}
Since $q \in \AC^\a_{a+} (I_f,\R^m)$, we have $q \in \L^1_\loc(I_f,\R^m)$ and $\I^{1-\a}_{a+} [q] \in \AC_\loc(I_f,\R^m)$. Moreover one can easily prove that $\I^{1-\a}_{a+} [q]_{|I} = \I^{1-\a}_{a+} [q_{|I}]$ on $I$. Thus $q_{|I} \in \L^1_\loc(I,\R^m)$ and $\I^{1-\a}_{a+} [q_{|I}] \in \AC_\loc(I,\R^m)$, that is $q_{|I} \in \AC^\a_{a+} (I,\R^m)$. Moreover $\I^{1-\a}_{a+} [q_{|I}](a) = \I^{1-\a}_{a+} [q](a) = q_a$ and $\Dm [q_{|I}] (t) =  \Dm[q](t) = f(q(t),t) = f_{|I}(q_{|I}(t),t)$ for almost every $t \in I$. 
\end{proof}

\begin{lemma}\label{lem2}
Let $b > a$ and $k \in \N$. The Bielecki norm defined on $\L^1([a,b],\R^m)$ by
$$ \Vert q \Vert_{1,k} := \int_a^b e^{-k(\t-a)} \vert q(\t) \vert_m \; d\t , $$
is equivalent to the classical norm $\Vert \cdot \Vert_1$. In particular, $\L^1([a,b],\R^m)$ endowed with the Bielecki norm $\Vert \cdot \Vert_{1,k}$ is complete.
\end{lemma}

\begin{proof}
Indeed it holds that
\begin{multline*}
\Vert q \Vert_{1,k} = \int_a^b e^{-k(\t-a)} \vert q(\t) \vert_m \; d\t \leq  \int_a^b \vert q(\t) \vert_m \; d\t = \Vert q \Vert_{1} \\
\leq  \int_a^b e^{k(b-\t)} \vert q(\t) \vert_m \; d\t = e^{k(b-a)} \int_a^b e^{-k(\t-a)} \vert q(\t) \vert_m \; d\t = e^{k(b-a)} \Vert q \Vert_{1,k} ,
\end{multline*}
for every $q \in \L^1([a,b],\R^m)$.
\end{proof}

\subsection{Proof of Theorem~\ref{thm1RL}}

We first prove Theorem~\ref{thm1RL} in the case where $I_f = [a,b]$ with $b > a$. Let $L$ be associated with $[a,b] \subset I_f$ in \eqref{HgloblipRL}. In that case $\L^1_\loc(I_f,\R^m) = \L^1([a,b],\R^m)$ and we endow $\L^1([a,b],\R^m)$ with the Bielecki norm $\Vert \cdot \Vert_{1,k}$ provided in Lemma~\ref{lem2} with $k \in \N^*$ sufficiently large in order to have $\ell := L \sum_{i=1}^m \frac{1}{k^{\a_i}} < 1$. Since $f$ satisfies~\eqref{HL1zero} and \eqref{HgloblipRL}, $f$ satisfies~\eqref{HL1}. As a consequence we can correctly define the application
$$ \fonction{F}{\L^1([a,b],\R^m)}{\L^1([a,b],\R^m)}{y}{\D^{1-\a}_{a+}[q_a]+\Im[f(y,\cdot)].} $$
From Proposition~\ref{prop1RL}, our aim is to prove that $F$ admits a unique fixed point. Let $y_1$, $y_2 \in \L^1([a,b],\R^m)$. From~\eqref{HgloblipRL} and from the classical Fubini theorem, we obtain
\begin{equation*}
\begin{array}{rcl}
\Vert F(y_2) - F(y_1) \Vert_{1,k} & = & \di \int_a^b e^{-k(\t-a)} \vert \Im [ f(y_2,\cdot)-f(y_1,\cdot) ](\t) \vert_m \; d\t \\
& = & \di \int_a^b e^{-k(\t-a)} \left\vert \int_a^\t \left[ \dfrac{1}{\Gamma (\a)} (\t-s)^{\a-1} \right] \otimes ( f(y_2(s),s)- f(y_1(s),s) ) \; ds \right\vert_m d\t \\
& \leq & \di \int_a^b e^{-k(\t-a)} \sum_{i=1}^m \left\vert \dfrac{1}{\Gamma(\a_i)} \int_a^\t (\t-s)^{\a_i -1} ( f_i(y_2(s),s)- f_i(y_1(s),s) ) \; ds \right\vert d\t \\
& \leq & \di  \sum_{i=1}^m \dfrac{1}{\Gamma(\a_i)} \int_a^b \int_a^\t e^{-k(\t-a)} (\t-s)^{\a_i -1} \vert f_i(y_2(s),s)- f_i(y_1(s),s) \vert \; ds \; d\t \\
& \leq & \di L \sum_{i=1}^m \dfrac{1}{\Gamma(\a_i)} \int_a^b \int_a^\t e^{-k(\t-a)} (\t-s)^{\a_i -1} \vert y_2(s)-y_1(s) \vert_m \; ds \; d\t \\
& \leq & \di L \sum_{i=1}^m \dfrac{1}{\Gamma(\a_i)} \int_a^b  \vert y_2(s)-y_1(s) \vert_m \int_s^b e^{-k(\t-a)} (\t-s)^{\a_i -1} \; d\t  \; ds.
\end{array}
\end{equation*}
On the other hand, it holds that
\begin{multline*}
\int_s^b e^{-k(\t-a)} (\t-s)^{\a_i -1} \; d\t = e^{-k(s-a)} \int_0^{b-s} e^{-k\t} \t^{\a_i -1} \; d\t \\
\leq  e^{-k(s-a)} \int_0^{+\infty} e^{-k\t} \t^{\a_i -1} \; d\t 
= \dfrac{e^{-k(s-a)}}{k^{\a_i}}  \int_0^{+\infty} e^{-u} u^{\a_i -1} \; du =  \dfrac{e^{-k(s-a)}}{k^{\a_i}} \Gamma (\a_i) .
\end{multline*}
Finally, we have proved that $F$ is a $\ell$-contraction map. It follows from the classical Banach fixed point theorem that $F$ has a unique fixed point.

\bigskip

Now let us prove Theorem~\ref{thm1RL} in the case where $I_f = [a,b)$ with $b > a$ and in the case where $I_f = [a,+\infty)$. In both cases, one can easily write $I_f = \cup_{p \in \N} I_p$ where $I_p := [a,b_p]$ and $(b_p)_p \subset I_f$ is an increasing sequence with $b_0 > a$. Let us denote by $f_p$ the restriction $f_{|I_p}$ of $f$. From the previous case and for any $p \in \N$, there exists a unique (global) solution $q_p : I_p \to \R^m$ to the restricted Cauchy problem~\eqref{eqcauchyproblemRLrestrictedn} given by
\begin{equation}\label{eqcauchyproblemRLrestrictedn}\tag{VCP${}_{p}$}
\left\lbrace \begin{array}{l}
\Dm [q](t) = f_p(q(t),t),  \\
\I^{1-\a}_{a+} [q](a) = q_a.
\end{array}
\right.
\end{equation}
From Lemma~\ref{lem1} and since ${f_{p+1}}_{|I_p} = f_p$, it clearly follows from the uniqueness of $q_p$ that $q_{p+1} = q_p$ almost everywhere on $I_p$. As a consequence, we can correctly define $q : I_f \to \R^m$ by $q(t) := q_p (t)$ if $t \in I_p$. Our aim is now to prove that $q$ is a (global) solution of~\eqref{eqcauchyproblemRL}. For any $b \in I_f$, there exists $p \in \N$ such that $[a,b] \subset I_p$ and then $q = q_p$ almost everywhere on $[a,b]$. As a consequence, one can easily conclude that $q \in \L^1_\loc (I_f,\R^m)$ and $\I^{1-\a}_{a+} [q] \in \AC_\loc (I_f,\R^m)$, that is $q \in \AC^\a_{a+} (I_f,\R^m)$, and $\I^{1-\a}_{a+}[q](a)=q_a$ and $\Dm [q] = f(q,\cdot)$ almost everywhere on $I_f$. Hence $q$ is a (global) solution of~\eqref{eqcauchyproblemRL}. By contradiction, let us assume that $q$ is not unique. Let $Q$ be another (global) solution of~\eqref{eqcauchyproblemRL}. From Lemma~\ref{lem1}, the restriction $Q_{|I_p}$ is then the unique (global) solution of~\eqref{eqcauchyproblemRLrestrictedn}, that is, $Q = q_p = q$ almost everywhere on $I_p$. Since this last equality is true for any $p \in \N$, we get that $Q = q$ almost everywhere on $I_f$ and the uniqueness is proved.

\section{Proofs of Section~\ref{sectionCauchyC}}\label{appsectionCauchyC}

\subsection{Proof of Proposition~\ref{prop1C}}
Since $f$ satisfies~\eqref{Hinfty}, note that $f(q,\cdot) \in \L^\infty_\loc (I,\R^m)$ for every couple $(q,I)$ such that $I \in \IImf$ and $q \in \C (I,\Omega)$. 

\smallskip

We first prove the necessary condition. Let $(q,I)$ be a local solution of~\eqref{eqcauchyproblemC}. Then $I \in \IImf$ and $q \in {}_\c \AC^\a_{a+}(I,\Omega) \subset \C(I,\Omega)$. Since $\I^{1-\a}_{a+} [q-q(a)] \in \AC_\loc(I,\R^m)$, it holds that
\begin{multline*}
\I^{1-\a}_{a+} [q-q(a)] =  \I^{1-\a}_{a+} [q-q(a)](a) + \I^1_{a+} \left[ \dfrac{d}{dt} \Big[  \I^{1-\a}_{a+} [q-q(a)] \Big] \right] \\
 =  \I^{1-\a}_{a+} [q-q(a)](a) + \I^1_{a+} [ \CDm [q] ]  =  \I^{1-\a}_{a+} [q-q(a)](a) + \I^1_{a+} [f(q,\cdot)] ,
\end{multline*}
everywhere on $I$. From Proposition~\ref{prop3} and since $q-q(a) \in \L^\infty_\loc(I,\R^m)$, it holds that $\I^{1-\a}_{a+} [q-q(a)](a) = 0$, even if $\a_i = 1$ for some $i =1,\ldots,n$. Finally we have proved that
$$ \I^{1-\a}_{a+} [q-q(a)] =  \I^1_{a+} [f(q,\cdot)] ,$$
everywhere on $I$. Since $q-q(a)$ and $f(q,\cdot) \in \L^\infty_\loc(I,\R^m)$, we obtain from Proposition~\ref{prop5} that
$$ \I^1_{a+} [q-q(a)] = \I^1_{a+} \Big[ \Im [f(q,\cdot) ] \Big], $$
everywhere on $I$. Since $f(q,\cdot) \in \L^\infty_\loc(I,\R^m)$, we have $\Im [f(q,\cdot) ] \in \C (I,\R^m)$ from Proposition~\ref{prop3}. Since $q-q(a)$ and $ \Im [f(q,\cdot)] \in \C(I,\R^m)$, differentiating the previous equality leads to
$$ q-q(a) = \Im [f(q,\cdot) ] ,$$
everywhere on $I$. 

\smallskip

Now let us prove the sufficient condition. Let us assume that $I \in \IImf$, $q \in \C(I,\Omega)$ and
$$ q(t) = q_a + \Im [f(q,\cdot)] (t) , $$
for every $t \in I$. Since $f(q,\cdot) \in \L^\infty_\loc(I,\R^m)$, we have $\Im [f(q,\cdot) ] \in \C_a (I,\R^m)$ from Proposition~\ref{prop3} and thus $q(a) = q_a$. Moreover, since $q-q(a)$ and $f(q,\cdot) \in \L^\infty_\loc(I,\R^m)$, we obtain from Proposition~\ref{prop5} that
$$ \I^{1-\a}_{a+} [q-q(a)] = \I^1_{a+} [f(q,\cdot)] , $$
everywhere on $I$. Since $f(q,\cdot) \in \L^\infty_\loc(I,\R^m)$, it clearly follows that $q \in {}_\c \AC^\a_\loc(I,\Omega)$ and
$$ \CDm [q] = \dfrac{d}{dt} \Big[ \I^{1-\a}_{a+} [q-q(a)] \Big] = f(q,\cdot) , $$
almost everywhere on $I$. We conclude that $(q,I)$ is a local solution of~\eqref{eqcauchyproblemC}.

\subsection{Proof of Theorem~\ref{thm1}}
The proof of Theorem~\ref{thm1} easily follows from the three following propositions.

\begin{proposition}
Every local solution of~\eqref{eqcauchyproblemC} can be extended to a maximal solution.
\end{proposition}

\begin{proof}
Let $(q,I)$ be a local solution of~\eqref{eqcauchyproblemC}. Let $\mathscr{F}$ be the nonempty set of all extensions of $(q,I)$ ordered by
\begin{equation*}
(q_1,I_1) \leq (q_2,I_2) \text{ if and only if } (q_2,I_2) \text{ is an extension of }  (q_1,I_1).
\end{equation*}
Our aim is to prove that $\mathscr{F}$ admits a maximal element. From the classical Zorn lemma, it is sufficient to prove that $\mathscr{F}$ is inductive. Let $\mathscr{G} = \{ (q_p,I_p) \}_{p \in \mathscr{P}}$ be a nonempty totally ordered subset of $\mathscr{F}$. Let us prove that $\mathscr{G}$ admits an upper bound in $\mathscr{F}$. Let us define $I' := \cup_{p \in \mathscr{P}} I_p$. Clearly $I' \in \IImf$. For every $t \in I'$, there exists $p \in \mathscr{P}$ such that $t \in I_p$ and, since $\mathscr{G}$ is totally ordered, if $t \in I_{p_1} \cap I_{p_2}$ then $q_{p_1} (t) = q_{p_2} (t)$. Consequently, we can (correctly) define $q' : I' \to \Omega$ by $q'(t) := q_p (t) \in \Omega$ if $ t \in I_p$. Similarly to the end of the proof of Theorem~\ref{thm1RL}, one can easily prove that $(q',I')$ is a local solution of~\eqref{eqcauchyproblemC}. Moreover $(q',I')$ extends $(q,I)$. As a consequence $(q',I') \in \mathscr{F}$ and is clearly an upper bound of $\mathscr{G}$. The proof is complete.
\end{proof}

\begin{proposition}\label{propdebase}
If $f$ satisfies \eqref{Hinfty} and \eqref{Hloclip}, then \eqref{eqcauchyproblemC} has a local solution.
\end{proposition}

\begin{proof}
Let $R$, $\delta$ and $L$ be associated with $(q_a,a) \in \Omega \times I_f$ in \eqref{Hloclip}. We assume that $\delta$ is sufficiently small in order to have $[a,a+\delta] \subset I_f$. Let $M$ be associated with $\BB_m(q_a,R) \in \KK$ and $[a,a+\delta]$ in \eqref{Hinfty}. Consider $0 < \varepsilon \leq \delta$ sufficiently small in order to have $ M \sum_{i=1}^m \frac{\varepsilon^{\a_i}}{\Gamma (1+\a_i)} \leq R$ and $\ell := L \sum_{i=1}^m \frac{\varepsilon^{\a_i}}{\Gamma (1+\a_i)} < 1$. Then we construct the $\ell$-contraction map given by
\begin{equation*}
\fonction{F}{\C ([a,a+\varepsilon],\BB_m(q_a,R))}{\C ([a,a+\varepsilon],\BB_m(q_a,R))}{y}{F(y),}
\end{equation*}
with
$$
\fonction{F(y)}{[a,a+\varepsilon]}{\BB_m(q_a,R)}{t}{q_a + \Im [ f(y,\cdot) ](t).}
$$
Indeed, from~\eqref{Hinfty} and Proposition~\ref{prop3}, we infer that $F(y) \in \C([a,a+\varepsilon],\R^m)$ for every $y \in \C ([a,a+\varepsilon],\BB_m(q_a,R))$. From~\eqref{Hinfty}, we claim that $\vert F(y)(t) - q_a \vert_m \leq R$ for every $y \in \C ([a,a+\varepsilon],\BB_m(q_a,R))$ and every $t \in [a,a+\varepsilon]$. Finally, from~\eqref{Hloclip}, we infer that $\Vert F(y_2) - F(y_1) \Vert_\infty \leq \ell \Vert y_2 - y_1 \Vert_\infty$ for every $y_1$, $y_2 \in \C ([a,a+\varepsilon],\BB_m(q_a,R))$. It follows from the classical Banach fixed point theorem that $F$ has a unique fixed point denoted by $q$. It follows from Proposition~\ref{prop1C} that $(q,[a,a+\varepsilon])$ is a local solution of \eqref{eqcauchyproblemC}. 
\end{proof}

\begin{proposition}
We assume that $f$ satisfies \eqref{Hinfty} and \eqref{Hloclip}. Let $(q,I)$ and $(q',I')$ be two local solutions of \eqref{eqcauchyproblemC}. If $I \subset I'$, then $(q',I')$ is an extension of $(q,I)$.
\end{proposition}

\begin{proof}
By contradiction let us assume that $ A :=  \{ t \in I, \; q' (t) \neq q (t) \} $ is not empty and let us consider $b := \inf A \in I$. Necessarily it holds that $q' = q$ on $[a,b]$ and $b < \sup I$. Let $R$, $\delta$ and $L$ be associated with $(q(b),b) \in \Omega \times I_f$ in \eqref{Hloclip}. We assume that $\delta$ is sufficiently small in order to have $[b,b+\delta] \subset I \subset I_f$. We introduce $z \in \C ([b,b+\delta],\R^m)$ given by
$$ \forall t \in [b,b+\delta], \quad z(t) := q_a +  \int_{a}^{b} \left[ \dfrac{1}{\Gamma (\alpha )} (t-\t)^{\alpha - 1} \right] \otimes f(q(\t),\t) \; d\t. $$
The continuity of $z$ can be proved from the classical Lebesgue dominated convergence theorem. Also note that $z(b) = q(b)$. Let $M$ be associated with $\BB_m(q(b),R) \in \KK$ and $[b,b+\delta]$ in \eqref{Hinfty}. Consider $0 < \varepsilon \leq \delta$ sufficiently small in order to have $ \ell := L \sum_{i=1}^m \frac{\varepsilon^{\a_i}}{\Gamma (1+\a_i)} < 1$ and
$$ \vert z(t) - q(b) \vert_m + M \sum_{i=1}^m \frac{\varepsilon^{\a_i}}{\Gamma (1+\a_i)} \leq R, \quad  \vert q(t) - q(b) \vert_m \leq R, \quad  \vert q' (t) - q(b) \vert_m \leq R , $$
for every $t \in [b,b+\varepsilon]$. Finally, as in the proof of Proposition~\ref{propdebase}, we consider the $\ell$-contraction map given by
\begin{equation*}
\fonction{F}{\C ([b,b+\varepsilon],\BB_m(q(b),R))}{\C ([b,b+\varepsilon],\BB_m(q(b),R))}{y}{F(y),}
\end{equation*}
with
$$
\fonction{F(y)}{[b,b+\varepsilon]}{\BB_m(q(b),R)}{t}{z(t) +  \di \int_{b}^t \left[ \dfrac{1}{\Gamma (\a) } (t-\t)^{\a -1} \right] \otimes f(y(\t),\t) \; d\t .}
$$
It follows from the classical Banach fixed point theorem that $F$ has a unique fixed point. Since $(q,I)$ and $(q',I')$ are local solutions of~\eqref{eqcauchyproblemC} and since $q = q'$ on $[a,b]$, one can easily prove that $q$ and $q'$ are fixed points of $F$. We conclude that $q' = q$ on $[b,b+\varepsilon]$ and then on $[a,b+\varepsilon]$. This raises a contradiction with the definition of $b$. Consequently $A$ is empty and the proof is complete.
\end{proof}

\subsection{Proof of Theorem~\ref{thm3}}
We first need to state the following lemma.

\begin{lemma}\label{lem2C}
Let $b > a$ and $k \in \N$. The Bielecki norm defined on $\C([a,b],\R^m)$ by
$$ \Vert q \Vert_{\infty,k} := \max_{t \in [a,b]} \vert e^{-k(t-a)} q(t) \vert_m , $$
is equivalent to the classical norm $\Vert \cdot \Vert_\infty$. In particular, $\C([a,b],\R^m)$ endowed with the Bielecki norm $\Vert \cdot \Vert_{\infty,k}$ is complete.
\end{lemma}

\begin{proof}
Indeed one can easily prove that
$$ \Vert q \Vert_{\infty,k} \leq \Vert q \Vert_{\infty} \leq e^{k(b-a)} \Vert q \Vert_{\infty,k} , $$
for every $q \in \C([a,b],\R^m)$.
\end{proof}

Now let us prove Theorem~\ref{thm3}. Since $f$ satisfies~\eqref{HgloblipC}, $f$ satisfies~\eqref{Hloclip}. From Theorem~\ref{thm1}, since $f$ also satisfies~\eqref{Hinfty}, \eqref{eqcauchyproblemC} admits a unique maximal solution denoted by $(q,I)$ and $(q,I)$ is the maximal extension of any other local solution of~\eqref{eqcauchyproblemC}. In order to prove that $(q,I)$ is global, it is then sufficient to prove that~\eqref{eqcauchyproblemC} admits a local solution $(Q,[a,b])$ for every $b \in I_f$, $b > a$.

\smallskip

Let $b \in I_f$ with $b > a$. Let $L$ be associated with $[a,b] \subset I_f$ in \eqref{HgloblipRL}. We endow $\C([a,b],\R^m)$ with the Bielecki norm $\Vert \cdot \Vert_{\infty,k}$ provided in Lemma~\ref{lem2C} with $k \in \N^*$ sufficiently large in order to have $\ell := L \sum_{i=1}^m \frac{1}{k^{\a_i}} < 1$. Then we consider 
\begin{equation*}
\fonction{F}{\C ([a,b],\R^m)}{\C ([a,b],\R^m)}{y}{F(y),}
\end{equation*}
with
$$
\fonction{F(y)}{[a,b]}{\R^m}{t}{q_a + \Im [f(y,\cdot)](t) .}
$$
Let $y_1$, $y_2 \in \C([a,b],\R^m)$. From~\eqref{HgloblipC} it holds that
\begin{eqnarray*}
\vert e^{-k(t-a)} (F(y_2)-F(y_1) )(t) \vert_m & = & \vert e^{-k(t-a)} \Im [ f(y_2,\cdot) - f(y_1,\cdot) ](t) \vert_m \\
& = & \left\vert e^{-k(t-a)} \int_a^t \left[ \dfrac{1}{\Gamma(\a)} (t-\t)^{\a-1} \right] \otimes (f(y_2(\t),\t) - f(y_1(\t),\t) ) \; d\t \right\vert_m \\
& \leq & \sum_{i=1}^m \left\vert  \dfrac{e^{-k(t-a)}}{\Gamma(\a_i)} \int_a^t  (t-\t)^{\a_i -1} (f_i(y_2(\t),\t) - f_i(y_1(\t),\t) ) \; d\t \right\vert \\
& \leq & \sum_{i=1}^m \dfrac{e^{-k(t-a)}}{\Gamma(\a_i)} \int_a^t  (t-\t)^{\a_i -1} \vert f_i(y_2(\t),\t) - f_i(y_1(\t),\t) ) \vert \; d\t \\
& \leq & L \sum_{i=1}^m \dfrac{e^{-k(t-a)}}{\Gamma(\a_i)} \int_a^t  (t-\t)^{\a_i -1} \vert y_2(\t) - y_1(\t) \vert_m \; d\t ,
\end{eqnarray*}
for every $t \in [a,b]$. On the other hand, note that $\vert y_2(\t) - y_1(\t) \vert_m \leq e^{k(\t-a)} \Vert y_2 - y_1 \Vert_{\infty,k}$ for every $\t \in [a,b]$. As a consequence, it holds that
$$ \vert e^{-k(t-a)} (F(y_2)-F(y_1) )(t) \vert_m \leq L \Vert y_2 - y_1 \Vert_{\infty,k} \sum_{i=1}^m \dfrac{1}{\Gamma(\a_i)} \int_a^t  (t-\t)^{\a_i -1} e^{-k(t-\t)} \; d\t , $$
for every $t \in [a,b]$. Similarly to the proof of Theorem~\ref{thm1RL}, one can easily prove that 
$$ \int_a^t  (t-\t)^{\a_i -1} e^{-k(t-\t)} \; d\t \leq \frac{\Gamma(\a_i)}{k^{\a_i}} .$$
Finally, we have proved that $F$ is a $\ell$-contraction map. It follows from the classical Banach fixed point theorem that $F$ has a unique fixed point denoted by $Q$. The proof is complete.

\subsection{Proof of Theorem~\ref{thm2}}

Theorem~\ref{thm2} corresponds to the last proposition of this section.

\begin{lemma}\label{appCprop21-6}
We assume that $f$ satisfies \eqref{Hinfty} and \eqref{Hloclip}. Let $(q,I)$ be the maximal solution of~\eqref{eqcauchyproblemC}. If $(q,I)$ is not global, then $I = [a,b)$ with $b \in I_f$, $b > a$. Moreover, $q$ cannot be continuously extended at $t=b$ with a $\Omega$-value.
\end{lemma}

\begin{proof}
Let us prove the first part of Lemma~\ref{appCprop21-6}. Precisely, we prove here that if $I = [a,b]$, then $b = \max I_f$ (and thus $I = I_f$). By contradiction let us assume that $I = [a,b]$ with $b < \sup I_f$. Let $R$, $\delta$ and $L$ be associated with $(q(b),b) \in \Omega \times I_f$ in \eqref{Hloclip}. We assume that $\delta$ is sufficiently small in order to have $[b,b+\delta] \subset I_f$. We introduce $z \in \C ([b,b+\delta],\R^m)$ given by
$$ \forall t \in [b,b+\delta], \quad z(t) := q_a +  \int_{a}^{b} \left[ \dfrac{1}{\Gamma (\alpha )} (t-\t)^{\alpha - 1} \right] \otimes f(q(\t),\t) \; d\t. $$
The continuity of $z$ can be proved from the classical Lebesgue dominated convergence theorem. Also note that $z(b) = q(b)$. Let $M$ be associated with $\BB_m(q(b),R) \in \KK$ and $[b,b+\delta]$ in \eqref{Hinfty}. Consider $0 < \varepsilon \leq \delta$ sufficiently small in order to have $\ell := L \sum_{i=1}^m \frac{\varepsilon^{\a_i}}{\Gamma (1+\a_i)} < 1$ and
$$ \vert z(t) - q(b) \vert_m + M \sum_{i=1}^m \frac{\varepsilon^{\a_i}}{\Gamma (1+\a_i)} \leq R,  $$
for every $t \in [b,b+\varepsilon]$. Finally, as in the proof of Proposition~\ref{propdebase}, we introduce a $\ell$-contraction map given by
\begin{equation*}
\fonction{F}{\C ([b,b+\varepsilon],\BB_m(q(b),R))}{\C ([b,b+\varepsilon],\BB_m(q(b),R))}{y}{F(y),}
\end{equation*}
with
$$
\fonction{F(y)}{[b,b+\varepsilon]}{\BB_m(q(b),R)}{t}{z(t) + \di \int_{b}^t \left[ \dfrac{1}{\Gamma (\a) }  (t-\t)^{\a -1} \right] \otimes f(y(\t),\t) \; d\t .}
$$
It follows from the classical Banach fixed point theorem that $F$ has a unique fixed point denoted by $Q$. One can easily prove that $q' : [a,b+\varepsilon] \to \Omega$ defined by
$$ q'(t) := \left\lbrace \begin{array}{lcl}
q(t) & \text{if} & t \in [a,b], \\
Q(t) & \text{if} & t \in [b,b+\varepsilon],
\end{array} \right. $$
is a local solution of~\eqref{eqcauchyproblemC} and is an extension of $(q,I)$ with $I \varsubsetneq [a,b+\varepsilon]$. This raises a contradiction with the maximality of $(q,I)$ and the proof of the first part is complete. 

\smallskip

Let us prove the second part of Lemma~\ref{appCprop21-6}. By contradiction let us assume that $q$ can be continuously extended at $t=b$ with a value $\xi \in \Omega$, that is, $\lim_{t \to b, \; t < b} q(t) = \xi \in \Omega$. Let $q' : [a,b] \to \Omega$ be the continuous function defined by
\begin{equation*}
q' (t) := \left\lbrace \begin{array}{lcl}
q(t) & \text{if} & t \in [a,b), \\
\xi & \text{if}& t=b.
\end{array} \right.
\end{equation*}
Our aim is to prove that $(q',[a,b])$ is a local solution of~\eqref{eqcauchyproblemC}. Since $(q,[a,b))$ is a local solution of~\eqref{eqcauchyproblemC}, it holds that
$$
q'(t) = q(t) = q_a +  \Im [f(q,\cdot)] (t) = q_a + \Im [f(q',\cdot)] (t) ,
$$
for every $t \in [a,b)$. Since $f(q',\cdot) \in \L^\infty ([a,b],\R^m)$, we infer from Proposition~\ref{prop4} that $\Im [f(q',\cdot)] \in \C([a,b],\R^m)$. From continuity, the above equality also holds true at $t=b$. It follows that $(q',[a,b])$ is a local solution of~\eqref{eqcauchyproblemC} and is an extension of $(q,[a,b))$ with $[a,b) \varsubsetneq [a,b]$, raising a contradiction with the maximality of $(q,[a,b))$. The proof is complete.
\end{proof}

\begin{proposition}
We assume that $f$ satisfies \eqref{Hinfty} and \eqref{Hloclip}. Let $(q,I)$ be the maximal solution of~\eqref{eqcauchyproblemC}. If $(q,I)$ is not global, then $I = [a,b)$ with $b \in I_f $, $b >a$, and, for every $K \in \KK$, there exists $t\in I$ such that $q(t) \notin K$.
\end{proposition}

\begin{proof}
The first part of this result is already proved in the first lemma of this section. By contradiction let us assume that there exists $K \in \KK $ such that $q (t) \in K$ for every $t \in [a,b)$. As a consequence, from~\eqref{Hinfty}, $f(q,\cdot) \in \L^\infty ([a,b),\R^m)$ and then $q \in \H^\a ([a,b),\R^m)$ from Proposition~\ref{prop4}. In particular, $q$ is uniformly continuous on $[a,b)$ and thus can be continuously extended at $t=b$ with a value $\xi \in \R^m$. Since $K$ is closed, we conclude that $\xi \in K \subset \Omega$. The proof is complete from the previous lemma.
\end{proof}

\section{Proofs of Section~\ref{sectransition}}\label{appsectransition}

\subsection{Proof of Lemma~\ref{lemimportant}}
Let us fix $b \in I$ with $b >a$. Since $A \in \L^\infty_\loc(I,\R^{m \times m})$, there exists $M \geq 0$ such that $\vert A_{ij} (\t) \vert \leq M$ for every $i$, $j \in \{ 1 , \ldots ,m \}$ and for almost every $\t \in [a,b]$.
Since $\a = (\a_{ij} ) \in (0,1]^{m \times m}$, we denote by $\beta := \min_{ij} \a_{ij} \in (0,1]$ and by $\gamma := \max_{ij} \a_{ij} \in (0,1]$. Finally, we denote by 
$$ \delta := \left\lbrace \begin{array}{lcr}
\beta & \text{if} & b-a < 1, \\
\gamma & \text{if} & b-a \geq 1.
\end{array} \right. $$
For every $i$, $j \in \{ 1 , \ldots ,m \}$, it follows from Definition~\ref{defRLstate} that
$$ 0 \leq \vert Z_{ij} (t,s) \vert \leq \dfrac{1}{\Gamma (\a_{ij})} (t-s)^{\a_{ij}-1} +M \sum_{k=1}^m \I^{\a_{ij}}_{s+} \big[ \; \vert Z_{kj}(\cdot,s) \vert \; \big](t) , $$
for almost every $a \leq s < t \leq b$. Now let us fix $i$, $j \in \{ 1 , \ldots ,m \}$. One can prove by induction that
\begin{multline*}
0 \leq \vert Z_{ij} (t,s) \vert \leq (t-s)^{\a_{ij} - 1} \left( \sum_{p=0}^{n-1} M^p \sum_{k_1,\ldots,k_p} \dfrac{1}{\Gamma (\a_{ij} + \sum_{q=1}^p \a_{k_q j})} (t-s)^{\sum_{q=1}^p \a_{k_q j}} \right) \\
+ M^n \sum_{k_1,\ldots,k_n} \I^{\a_{ij} + \sum_{q=1}^{n-1} \a_{k_q j}}_{s+} \big[ \; \vert Z_{k_n j }(\cdot,s) \vert \; \big] (t),
\end{multline*}
for almost every $a \leq s < t \leq b$ and for every $n \in \N^*$. Thus
\begin{multline*}
0 \leq \vert Z_{ij} (t,s) \vert \leq (t-s)^{\a_{ij} - 1} \left( \sum_{p=0}^{n-1} ( M (b-a)^{\delta} )^p \sum_{k_1,\ldots,k_p} \dfrac{1}{\Gamma (\a_{ij} + \sum_{q=1}^p \a_{k_q j})} \right) \\
+ M^n \sum_{k_1,\ldots,k_n} \I^{\a_{ij} + \sum_{q=1}^{n-1} \a_{k_q j}}_{s+} \big[ \; \vert Z_{k_n j }(\cdot,s) \vert \; \big] (t),
\end{multline*}
for almost every $a \leq s < t \leq b$ and for every $n \in \N^*$. For $p \in \N$ sufficiently large, we have $ (p+1) \beta \geq 2$ and thus
$$ \sum_{k_1,\ldots,k_p} \dfrac{1}{\Gamma (\a_{ij} + \sum_{q=1}^p \a_{k_q j})} \leq \dfrac{m^p}{\Gamma ((p+1)\beta)}. $$
From the definition of the classical Mittag-Leffler function, we conclude that the series
$$  \sum_{p=0}^{n-1} ( M (b-a)^{\delta} )^p \sum_{k_1,\ldots,k_p} \dfrac{1}{\Gamma (\a_{ij} + \sum_{q=1}^p \a_{k_q j})} , $$
converges to some $\Theta^{b}_{ij} \in \R^+$ when $n \to \infty$. Now let us assume that $n \in \N^*$ is sufficiently large in order to have $n \beta \geq 2$. Then, for almost every $a \leq s < t \leq b$, it holds that
$$
M^n \sum_{k_1,\ldots,k_n} \I^{\a_{ij} + \sum_{q=1}^{n-1} \a_{k_q j}}_{s+} \big[ \; \vert Z_{k_n j }(\cdot,s) \vert \; \big] (t) \leq \dfrac{M^n m^{n-1} (b-a)^{n\delta -1}}{\Gamma (n\beta)} \sum_{k=1}^m \int_s^t \vert Z_{kj} (\t,s) \vert \; d\t ,
$$
that tends to zero when $n \to \infty$. Finally, we have proved that
$$ 0 \leq \vert Z_{ij} (t,s) \vert \leq (t-s)^{\a_{ij}-1} \Theta^{b}_{ij}, $$
for almost every $a \leq s < t \leq b$. To conclude, one has to define $\Theta^b := \max_{ij} \Theta^b_{ij}$.

\subsection{Proof of Theorem~\ref{thmduhamelRL}}
From Lemma~\ref{lemimportant}, we can correctly define the function $q \in \L^1_\loc(I,\R^m)$ by
$$ q(t) := Z(t,a) \times q_a + \int_a^t Z(t,s) \times B(s) \; ds , $$
for almost every $t \in I$. Let us prove that $q$ is the unique (global) solution of~\eqref{eqcauchyproblemRLLV} provided in virtue of Theorem~\ref{thm1RL}. From Definition~\ref{defRLstate} it holds that
\begin{multline*}
q(t)  =  \left( \left[ \dfrac{1}{\Gamma(\aa)} (t-a)^{\aa-1} \right] \otimes \ID_m + \int_a^t \left[ \dfrac{1}{\Gamma(\aa)} (t-\t)^{\aa-1} \right] \otimes \Big[ A(\t) \times Z(\t,a) \Big] \; d\t \right) \times q_a \\
+ \int_a^t \left( \left[ \dfrac{1}{\Gamma(\aa)} (t-s)^{\aa-1} \right] \otimes \ID_m + \int_s^t \left[ \dfrac{1}{\Gamma(\aa)} (t-\t)^{\aa-1} \right] \otimes \Big[ A(\t) \times Z(\t,s) \Big] \; d\t \right) \times B(s) \; ds,
\end{multline*}
for almost every $t \in I$. It follows from Lemma~\ref{lemcalcul} that
\begin{multline*}
q(t)  =  \left[ \dfrac{1}{\Gamma(\a)} (t-a)^{\a-1} \right] \otimes q_a + \int_a^t \left[ \dfrac{1}{\Gamma(\a)} (t-\t)^{\a-1} \right] \otimes \Big[ A(\t) \times Z(\t,a) \times q_a + B(\t) \Big] \; d\t \\
+ \int_a^t \int_s^t \left[ \dfrac{1}{\Gamma(\a)} (t-\t)^{\a-1} \right] \otimes \Big[ A(\t) \times Z(\t,s) \times B(s) \Big] \; d\t \; ds,
\end{multline*}
for almost every $t \in I$. From the classical Fubini formula it holds that
\begin{multline*}
\int_a^t \int_s^t \left[ \dfrac{1}{\Gamma(\a)} (t-\t)^{\a-1} \right] \otimes \Big[ A(\t) \times Z(\t,s) \times B(s) \Big] \; d\t \; ds \\
= \int_a^t \left[ \dfrac{1}{\Gamma(\a)} (t-\t)^{\a-1} \right] \otimes  \left[ A(\t) \times \int_a^\t Z(\t,s) \times B(s) \; ds \right] \; d\t.
\end{multline*}
Combining the two previous equalities we conclude that
$$ q(t) = \D^{1-\a}_{a+} [q_a] (t) + \I^\a_{a+} [ A \times q + B](t), $$
for almost every $t \in I$. We conclude from Proposition~\ref{prop1RL} that $q$ is the unique (global) solution of~\eqref{eqcauchyproblemRLLV}.

\subsection{Proof of Theorem~\ref{thmduhamelC}}
This proof is very similar to the proof of Theorem~\ref{thmduhamelRL}. From Lemma~\ref{lemimportant}, we can correctly define the function $q \in \C(I,\R^m)$ by
$$ q(t) := {}_\c Z(t,a) \times q_a + \int_a^t Z(t,s) \times B(s) \; ds , $$
for every $t \in I$. Let us prove that $q$ is the unique maximal solution, that is moreover global, of~\eqref{eqcauchyproblemCLV} provided in virtue of Theorems~\ref{thm1} and~\ref{thm3}. From Definitions~\ref{defRLstate} and \ref{defCstate} it holds that
\begin{multline*}
q(t)  =  \left( \ID_m + \int_a^t \left[ \dfrac{1}{\Gamma(\aa)} (t-\t)^{\aa-1} \right] \otimes \Big[ A(\t) \times {}_\c Z(\t,a) \Big] \; d\t \right) \times q_a \\
+ \int_a^t \left( \left[ \dfrac{1}{\Gamma(\aa)} (t-s)^{\aa-1} \right] \otimes \ID_m + \int_s^t \left[ \dfrac{1}{\Gamma(\aa)} (t-\t)^{\aa-1} \right] \otimes \Big[ A(\t) \times Z(\t,s) \Big] \; d\t \right) \times B(s) \; ds,
\end{multline*}
for every $t \in I$. It follows from Lemma~\ref{lemcalcul} that
\begin{multline*}
q(t)  =  q_a + \int_a^t \left[ \dfrac{1}{\Gamma(\a)} (t-\t)^{\a-1} \right] \otimes \Big[ A(\t) \times {}_\c Z(\t,a) \times q_a + B(\t) \Big] \; d\t \\
+ \int_a^t \int_s^t \left[ \dfrac{1}{\Gamma(\a)} (t-\t)^{\a-1} \right] \otimes \Big[ A(\t) \times Z(\t,s) \times B(s) \Big] \; d\t \; ds,
\end{multline*}
for every $t \in I$. From the classical Fubini formula it holds that
\begin{multline*}
\int_a^t \int_s^t \left[ \dfrac{1}{\Gamma(\a)} (t-\t)^{\a-1} \right] \otimes \Big[ A(\t) \times Z(\t,s) \times B(s) \Big] \; d\t \; ds \\
= \int_a^t \left[ \dfrac{1}{\Gamma(\a)} (t-\t)^{\a-1} \right] \otimes  \left[ A(\t) \times \int_a^\t Z(\t,s) \times B(s) \; ds \right] \; d\t.
\end{multline*}
Combining the two previous equalities we conclude that
$$ q(t) = q_a + \I^\a_{a+} [ A \times q + B](t), $$
for every $t \in I$. We conclude from Proposition~\ref{prop1C} that $q$ is the unique maximal solution, that is moreover global, of~\eqref{eqcauchyproblemCLV}.

\subsection{Proof of Theorem~\ref{thmduality}}
Since $Z(\cdot,\cdot)$ is the left R-L state-transition matrix associated to $A$ and $\aa$, it follows from Definition~\ref{defRLstate} that
$$ Z(t,s) = \left[ \dfrac{1}{\Gamma(\aa)} (t-s)^{\aa-1} \right] \otimes \ID_m + \int_s^t \left[ \dfrac{1}{\Gamma(\aa)} (t-\t)^{\aa-1} \right] \otimes \Big[ A(\t) \times Z(\t,s) \Big] \; d\t , $$
for almost every $t$, $s \in I$ with $t >s$. From Lemma~\ref{lemimportant}, we can correctly define $T(t,s)$ by
$$ T(t,s) := \left[ \dfrac{1}{\Gamma(\aaa)} (t-s)^{\aaa-1} \right] \otimes \ID_m + \int_s^t \left[ \dfrac{1}{\Gamma(\aaa)} (\t-s)^{\aaa-1} \right] \otimes \Big[ Z(t,\t) \times A(\t) \Big] \; d\t , $$
for almost every $t$, $s \in I$ with $t > s$. Our aim is prove that $Z(\cdot,\cdot) = T(\cdot,\cdot)$. It follows from Definition~\ref{defRLstate} that
\begin{multline*}
\int_s^t \left[ \dfrac{1}{\Gamma(\aaa)} (\t-s)^{\aaa-1} \right] \otimes \Big[ Z(t,\t) \times A(\t) \Big] \; d\t \\
= \int_s^t \left[ \dfrac{1}{\Gamma(\aaa)} (\t-s)^{\aaa-1} \right] \otimes \left[ \left( \Big[   \dfrac{1}{\Gamma(\aa)} (t-\t)^{\aa-1}  \Big] \otimes \ID_m  \right) \times A(\t) \right] \; d\t \\
+ \int_s^t \left[ \dfrac{1}{\Gamma(\aaa)} (\t-s)^{\aaa-1} \right] \otimes \left[ \left( \int_\t^t \left[ \dfrac{1}{\Gamma(\aa)} (t-\xi)^{\aa-1} \right] \otimes \Big[ A(\xi) \times Z(\xi,\t) \Big] \; d\xi \right) \times A(\t) \right] \; d\t ,
\end{multline*}
for almost every $t$, $s \in I$ with $t > s$. From the classical Fubini formula and from Lemmas~\ref{lemcalcul2} and \ref{lemcalcul3}, we prove that
\begin{multline*}
\int_s^t \left[ \dfrac{1}{\Gamma(\aaa)} (\t-s)^{\aaa-1} \right] \otimes \Big[ Z(t,\t) \times A(\t) \Big] \; d\t \\
= \int_s^t \left[ \dfrac{1}{\Gamma(\aa)} (t-\t)^{\aa-1} \right] \otimes \left[ A(\t) \times \left( \Big[   \dfrac{1}{\Gamma(\aaa)} (\t-s)^{\aaa-1}  \Big] \otimes \ID_m  \right) \right] \; d\t \\
+ \int_s^t \left[ \dfrac{1}{\Gamma(\aa)} (t-\xi)^{\aa-1} \right] \otimes \left[ A(\xi) \times \left( \int_s^\xi \left[ \dfrac{1}{\Gamma(\aaa)} (\t-s)^{\aaa-1} \right] \otimes \Big[  Z(\xi,\t) \times A(\t) \Big] \; d\t \right) \right] \; d\xi ,
\end{multline*}
for almost every $t$, $s \in I$ with $t > s$. Finally we obtain
\begin{multline*}
\int_s^t \left[ \dfrac{1}{\Gamma(\aaa)} (\t-s)^{\aaa-1} \right] \otimes \Big[ Z(t,\t) \times A(\t) \Big] \; d\t \\
= \int_s^t \left[ \dfrac{1}{\Gamma(\aa)} (t-\t)^{\aa-1} \right] \otimes \left[ A(\t) \times \left( \Big[   \dfrac{1}{\Gamma(\aaa)} (\t-s)^{\aaa-1}  \Big] \otimes \ID_m  \right. \right. \\
+ \left. \left. \int_s^\t \left[ \dfrac{1}{\Gamma(\aaa)} (\xi-s)^{\aaa-1} \right] \otimes \Big[  Z(\t,\xi) \times A(\xi) \Big] \; d\xi \right) \right] \; d\t ,
\end{multline*}
for almost every $t$, $s \in I$ with $t > s$. We finally use Lemma~\ref{lemcalcul4} in order to conclude that
$$ T(t,s) = \left[ \dfrac{1}{\Gamma(\aa)} (t-s)^{\aa-1} \right] \otimes \ID_m + \int_s^t \left[ \dfrac{1}{\Gamma(\aa)} (t-\t)^{\aa-1} \right] \otimes \Big[ A(\t) \times T(\t,s) \Big] \; d\t ,$$
for almost every $t$, $s \in I$ with $t > s$. From the definition and the uniqueness of $Z(\cdot,\cdot)$, we obtain that $T(\cdot,\cdot) = Z(\cdot,\cdot)$ and the proof is complete.

\subsection{Proof of Theorem~\ref{thmduality2}}
This proof is very similar to the proof of Theorem~\ref{thmduality}. Since $A(\cdot) = A$ is constant and since ${}_\c Z(\cdot,\cdot)$ is the left Caputo state-transition matrix associated to $A$ and $\aa$, it follows from Definition~\ref{defCstate} that
$$ {}_\c Z(t,s) = \ID_m + \int_s^t \left[ \dfrac{1}{\Gamma(\aa)} (t-\t)^{\aa-1} \right] \otimes \Big[ A \times {}_\c Z(\t,s) \Big] \; d\t , $$
for every $t$, $s \in I$ with $t \geq s$. We define $T(t,s)$ by
$$ T(t,s) := \ID_m + \int_s^t \left[ \dfrac{1}{\Gamma(\aaa)} (\t-s)^{\aaa-1} \right] \otimes \Big[ {}_\c Z(t,\t) \times A \Big] \; d\t , $$
for every $t$, $s \in I$ with $t \geq s$. Our aim is prove that ${}_\c Z(\cdot,\cdot) = T(\cdot,\cdot)$. It follows from Definition~\ref{defCstate} that
\begin{multline*}
\int_s^t \left[ \dfrac{1}{\Gamma(\aaa)} (\t-s)^{\aaa-1} \right] \otimes \Big[ {}_\c Z(t,\t) \times A \Big] \; d\t
= \int_s^t \left[ \dfrac{1}{\Gamma(\aaa)} (\t-s)^{\aaa-1} \right] \otimes A \; d\t  \\
+ \int_s^t \left[ \dfrac{1}{\Gamma(\aaa)} (\t-s)^{\aaa-1} \right] \otimes \left[ \left( \int_\t^t \left[ \dfrac{1}{\Gamma(\aa)} (t-\xi)^{\aa-1} \right] \otimes \Big[ A \times {}_\c Z(\xi,\t) \Big] \; d\xi \right) \times A \right] \; d\t ,
\end{multline*}
for every $t$, $s \in I$ with $t \geq s$. With the help of a change of variable (and since $A(\cdot) = A$ is constant), from the classical Fubini formula and from Lemmas~\ref{lemcalcul2} and \ref{lemcalcul3}, we prove that
\begin{multline*}
\int_s^t \left[ \dfrac{1}{\Gamma(\aaa)} (\t-s)^{\aaa-1} \right] \otimes \Big[ {}_\c Z(t,\t) \times A \Big] \; d\t
= \int_s^t \left[ \dfrac{1}{\Gamma(\aa)} (t-\t)^{\aa-1} \right] \otimes A \; d\t \\
+ \int_s^t \left[ \dfrac{1}{\Gamma(\aa)} (t-\xi)^{\aa-1} \right] \otimes \left[ A \times \left( \int_s^\xi \left[ \dfrac{1}{\Gamma(\aaa)} (\t-s)^{\aaa-1} \right] \otimes \Big[  {}_\c Z(\xi,\t) \times A \Big] \; d\t \right) \right] \; d\xi ,
\end{multline*}
for every $t$, $s \in I$ with $t \geq s$. Finally we obtain
\begin{multline*}
\int_s^t \left[ \dfrac{1}{\Gamma(\aaa)} (\t-s)^{\aaa-1} \right] \otimes \Big[ {}_\c Z(t,\t) \times A \Big] \; d\t \\
= \int_s^t \left[ \dfrac{1}{\Gamma(\aa)} (t-\t)^{\aa-1} \right] \otimes \left[ A \times \left( \ID_m + \int_s^\t \left[ \dfrac{1}{\Gamma(\aaa)} (\xi-s)^{\aaa-1} \right] \otimes \Big[  {}_\c Z(\t,\xi) \times A \Big] \; d\xi \right) \right] \; d\t ,
\end{multline*}
for every $t$, $s \in I$ with $t \geq s$. We actually have obtained that
$$ T(t,s) = \ID_m + \int_s^t \left[ \dfrac{1}{\Gamma(\aa)} (t-\t)^{\aa-1} \right] \otimes \Big[ A \times T(\t,s) \Big] \; d\t ,$$
for every $t$, $s \in I$ with $t \geq s$. From the definition and the uniqueness of ${}_\c Z(\cdot,\cdot)$, we conclude that $T(\cdot,\cdot) = {}_\c Z(\cdot,\cdot)$ and the proof is complete.

\bibliographystyle{plain}
%\bibliography{../bibliototal}

\end{document}